\documentclass[12pt]{article}
\usepackage{srcltx}
\usepackage{eurosym}
\usepackage{mathtools}
\usepackage{amsmath}
\usepackage{amsfonts}
\usepackage{amssymb}
\usepackage{amsthm}
\usepackage{graphicx}
\usepackage{mathrsfs}
\usepackage{xcolor}
\usepackage{exscale}
\usepackage{latexsym}
\usepackage{authblk}

\usepackage[colorlinks,plainpages=true,pdfpagelabels,hypertexnames=true,colorlinks=true,pdfstartview=FitV,linkcolor=blue,citecolor=red,urlcolor=black]{hyperref}
\PassOptionsToPackage{unicode}{hyperref}
\PassOptionsToPackage{naturalnames}{hyperref}
\usepackage{enumerate}
\usepackage[shortlabels]{enumitem}
\usepackage{bookmark}
\usepackage{wasysym}
\usepackage{esint}
\usepackage[ddmmyyyy]{datetime}
\usepackage[margin=1in]{geometry}
\makeatletter
\g@addto@macro\normalsize{%
	\setlength\abovedisplayskip{4pt}
	\setlength\belowdisplayskip{4pt}
	\setlength\abovedisplayshortskip{4pt}
	\setlength\belowdisplayshortskip{4pt}
}
\numberwithin{equation}{section}
\everymath{\displaystyle}
\usepackage[capitalize,nameinlink]{cleveref}
\crefname{section}{Section}{Sections}
\crefname{subsection}{Subsection}{Subsections}
\crefname{condition}{Condition}{Conditions}
\crefname{hypothesis}{Hypothesis}{Conditions}
\crefname{assumption}{Assumption}{Assumptions}
\crefname{lemma}{Lemma}{Lemmas}
\crefname{definition}{Definition}{Definitions}

\crefformat{equation}{\textup{#2(#1)#3}}
\crefrangeformat{equation}{\textup{#3(#1)#4--#5(#2)#6}}
\crefmultiformat{equation}{\textup{#2(#1)#3}}{ and \textup{#2(#1)#3}}
{, \textup{#2(#1)#3}}{, and \textup{#2(#1)#3}}
\crefrangemultiformat{equation}{\textup{#3(#1)#4--#5(#2)#6}}%
{ and \textup{#3(#1)#4--#5(#2)#6}}{, \textup{#3(#1)#4--#5(#2)#6}}%
{, and \textup{#3(#1)#4--#5(#2)#6}}

\Crefformat{equation}{#2Equation~\textup{(#1)}#3}
\Crefrangeformat{equation}{Equations~\textup{#3(#1)#4--#5(#2)#6}}
\Crefmultiformat{equation}{Equations~\textup{#2(#1)#3}}{ and \textup{#2(#1)#3}}
{, \textup{#2(#1)#3}}{, and \textup{#2(#1)#3}}
\Crefrangemultiformat{equation}{Equations~\textup{#3(#1)#4--#5(#2)#6}}%
{ and \textup{#3(#1)#4--#5(#2)#6}}{, \textup{#3(#1)#4--#5(#2)#6}}%
{, and \textup{#3(#1)#4--#5(#2)#6}}

\crefdefaultlabelformat{#2\textup{#1}#3}

\newtheorem{theorem} {Theorem}[section]
\newtheorem{proposition} [theorem]{Proposition}

\newtheorem{lemma}[theorem]{Lemma}
\newtheorem{corollary}[theorem]{Corollary}

\newtheorem{counter example}[theorem]{Counter Example}
\newtheorem{remark}[theorem] {Remark}

\def\CC{{\rm \kern.24em \vrule width.02em height1.4ex depth-.05ex \kern-.26emC}}

\def\TagOnRight

\def\AA{{it I} \hskip-3pt{\tt A}}

\def\QQ{\rlap {\raise 0.4ex \hbox{$\scriptscriptstyle |$}} {\hskip -0.1em Q}}

\makeatletter
\newcommand{\vo}{\vec{o}\@ifnextchar{^}{\,}{}}
\makeatother

\def\YYint#1#2#3{{\setbox0=\hbox{$#1{#2#3}{\iint}$}
		\vcenter{\hbox{$#2#3$}}\kern-.50\wd0}}

\def\XXint#1#2#3{{\setbox0=\hbox{$#1{#2#3}{\int}$}
		\vcenter{\hbox{$#2#3$}}\kern-.50\wd0}}

\makeatletter
\def\namedlabel#1#2{\begingroup
	\def\@currentlabel{#2}%
	\label{#1}\endgroup
}
\makeatother
\makeatletter
\newcommand{\rmh}[1]{\mathpalette{\raisem@th{#1}}}
\newcommand{\raisem@th}[3]{\hspace*{-1pt}\raisebox{#1}{$#2#3$}}
\makeatother

\newcounter{desccount}
\newcommand{\descitem}[2]{\item[#1]\refstepcounter{desccount}\label{#2}}
\newcommand{\descref}[2]{\hyperref[#1]{\textnormal{\textcolor{black}{}\textcolor{blue}{ #2}\textcolor{black}{}}}}

\newcommand{\dref}[2]{\hyperref[#1]{\textcolor{black}{(}\textcolor{blue}{\bf #2}\textcolor{black}{)}}}
\newcommand{\be} {\begin{eqnarray}}
	\newcommand{\ee} {\end{eqnarray}}
\newcommand{\Bea} {\begin{eqnarray*}}
	\newcommand{\Eea} {\end{eqnarray*}}
\newcommand{\pa} {\partial}

\newcommand{\al} {\alpha}
\newcommand{\rr}{\rightarrow}

\newcommand{\B} {\beta}
\newcommand{\de} {\delta}

\newcommand{\p}  {\prime}
\newcommand{\e}  {\varepsilon}

\newcommand{\la} {\lambda}
\newcommand{\si} {\sigma}

\newcommand{\f}{\infty}
\newcommand{\R}{\mathbb{R}}

\newcommand{\noi} {\noindent}

\newcommand{\ta}{\tau}

\newcommand{\norm}[1]{\left|\hspace{-0.2mm}\left| #1 \right|\hspace{-0.2mm}\right|}
\newcommand{\abs}[1]{\left| #1\right|}

\newcounter{whitney}
\refstepcounter{whitney}

\newcounter{ineqcounter}
\refstepcounter{ineqcounter}
\makeatletter
\def\ps@pprintTitle{%
	\let\@oddhead\@empty
	\let\@evenhead\@empty
	\def\@oddfoot{}%
	\let\@evenfoot\@oddfoot}
\makeatother
\usepackage[doublespacing]{setspace}
\usepackage[titletoc,toc,page]{appendix}

\makeatletter
\newcommand{\refcheckize}[1]{%
	\expandafter\let\csname @@\string#1\endcsname#1%
	\expandafter\DeclareRobustCommand\csname relax\string#1\endcsname[1]{%
		\csname @@\string#1\endcsname{##1}\wrtusdrf{##1}}%
	\expandafter\let\expandafter#1\csname relax\string#1\endcsname
}
\makeatother

\refcheckize{\cref}
\refcheckize{\Cref}


\makeatletter
\newcommand{\mainsectionstyle}{%
	\renewcommand{\@secnumfont}{\bfseries}
	\renewcommand\section{\@startsection{section}{2}%
		\z@{.5\linespacing\@plus.7\linespacing}{-.5em}%
		{\normalfont\bfseries}}%
}
\makeatother
\usepackage{pgf,tikz}
\usetikzlibrary{arrows}
\usetikzlibrary{decorations.pathreplacing}

\usepackage{xpatch}
\xpatchcmd{\MaketitleBox}{\hrule}{}{}{}
\xpatchcmd{\MaketitleBox}{\hrule}{}{}{}

\date{}

\usepackage{scalerel}

\makeatletter

\title{Vanishing viscosity limit for hyperbolic system of Temple class in 1-d with nonlinear viscosity}

\author[1,a]{Boris Haspot}
\author[2,a]{Animesh Jana}

\affil[a]{\footnotesize	 Universit\'e Paris Dauphine, PSL Research University, CEREMADE (UMR CNRS 7534), Place du Mar\' echal De Lattre De Tassigny 75775 Paris cedex 16 (France).}

\affil[1]{\em \footnotesize	 haspot@ceremade.dauphine.fr}
\affil[2]{\em \footnotesize	 animesh.jana@dauphine.psl.eu}

\allowdisplaybreaks

\begin{document}
	\maketitle
	\begin{abstract}
		We consider hyperbolic system with nonlinear viscosity such that the viscosity matrix $B(u)$ is commutating with $A(u)$ the matrix associated to the convective term. The drift matrix is assumed to be Temple class.  First we prove the global existence of smooth solutions for initial data with small total variation. We show that the solution to the parabolic equation converges to a semi-group solution of the hyperbolic system as viscosity goes to zero. Furthermore, we prove that the zero diffusion limit coincides with the one obtained in [Bianchini and Bressan, Indiana Univ. Math. J. 2000].
	\end{abstract}
	\section{Introduction}
	This article concerns about vanishing viscosity limit for hyperbolic system of conservation laws. We consider the following parabolic approximation of the hyperbolic system 
	
	\begin{align}
		u_t+A(u)u_x&=\varepsilon(B(u)u_x)_x \quad \quad\,\mbox{ for }t>0,x\in\R,\label{eqn-parabolic}\\
		u(0,x)&=u_0(x)\quad \quad\quad \quad\quad\mbox{for }x\in\R,\label{eqn:intial-data}
	\end{align}
	
where $u:[0,\f)\times\R\rr\R^n$ and $A,B$ are $n\times n$ matrices satisfying the following conditions for some $\mathcal{U}\subset\R^n$.
\begin{description}
	\descitem{($\mathcal{H}_A1$.)}{A1} The matrix $A(u)$ is $C^3$ function and has $n$ distinct eigenvalues $\la_1(u)<\cdots<\la_n(u)$ for $u\in\mathcal{U}$. 
	\descitem{($\mathcal{H}_B$.)}{B} The matrix $B(u)$ is a $C^2$ function and positive symmetric definite with $B(u)\geq c_0\mathbb{I}_n$ for $u\in\mathcal{U}$ and for some $c_0>0$. 

\end{description}
In this article, we prove the global existence of solutions to \eqref{eqn-parabolic} and we consider the vanishing viscosity limit as $\e\rr0$ for initial data having small total variation. We show that the limit converges to semigroup solution of the following partial differential equation,
\begin{equation}\label{eqn:hyperbolic}
	u_t+A(u)u_x=0.
\end{equation}
When $A=Df$ for some $C^1$ function $f:\Omega\rr\R^n$ and $\Omega\subset\R^n$, we may write \eqref{eqn:hyperbolic} in divergence form
\begin{equation}\label{eqn:conlaw}
	u_t+(f(u))_x=0.
\end{equation}
For strictly hyperbolic system of conservation laws \eqref{eqn:conlaw}, the existence of global BV solutions of hyperbolic system of conservation laws has been proved in \cite{Glimm} for initial data with small total variation. The uniqueness of entropy solutions has been established in \cite{Bressan,BrCrPi,BrLiuYa}. It was an open question to show that these entropy solutions can be obtained as a viscosity limit. For $n\times n$ hyperbolic system of conservation laws this question has been sloved in \cite{BiB-vv-lim} when $B(u)=\mathbb{I}_n$. We also refer to \cite{BiB-triangular,BiB-DCDS} for relevant results in this direction. 

This motivates us to study the vanishing viscosity limit for general viscosity matrix $B$ which may not be $\mathbb{I}_n$. It has been mentioned \cite{BiB-vv-lim,Bressan-notes} as an open problem to study the viscosity limit for general $B$ since there are many physical systems for which the viscosity matrix $B$ is quasilinear. Despite of physical importance, there are only a few results considering a non identity viscosity matrix. In \cite{BiSpinolo-arma} the boundary Riemann problem has been studied for system of conservation laws via vanishing viscosity process even with non-invertible or degenerate $B$ (see 
also \cite{BiSpinolo-JDE,BiSpinolo-ReIsMaUnTr} for further discussions in this direction). We also refer to \cite{Christoforou} where the viscosity limit has been studied for balance laws with dissipative forcing. 

The directional derivative of a function $g:\R^n\rr\R$ is denoted by $\zeta\bullet g$ for some $\zeta\in \R^n$. More precisely, 
\begin{equation*}
	\zeta\bullet g(u)=\lim\limits_{z\rr0}\frac{g(u+z\zeta)-g(u)}{z}.
\end{equation*} In this article, we study the vanishing viscosity limit for a larger class of matrix $B$ which is satisfying \descref{B}{($\mathcal{H}_B$.)} and which is commutating with $A(u)$. In addition the drift matrix $A$ satisfies the following assumption
\begin{description}
	\descitem{($\mathcal{H}_A2$.)}{A2} Each characteristic field of $A$ belongs to the Temple class, that is, the characteristic fields are straight lines and satisfying 
	\begin{equation}
		r_i\bullet r_i=0\mbox{ for all }i=1,2,\cdots,n,
	\end{equation} 
where $\{r_i\}_{i=1}^{n}$ are right eigenvectors fo $A(u)$.
\end{description}
Note that \descref{A2}{($\mathcal{H}_A2$.)} is equivalent to the property that shock and rarefaction curves coincide (see \cite[Theorem 2]{Temple}). For Temple class system of conservation laws \eqref{eqn:conlaw} with $A=Df$ satisfying \descref{A2}{($\mathcal{H}_A2$.)}, the existence and uniqueness is known for initial data with large total variation \cite{Serre-1,Bai-Bre}. In \cite{Serre-1} Serre has studied the vanishing viscosity limit in the case of a $2\times2$ Temple system for small $BV$ initial data with $B=\mathbb{I}_n$. In \cite{BiB-temple-class} Bianchini and Bressan have proved  the global existence of smooth solutions to \eqref{eqn-parabolic} for Temple $n\times n$ system and  small initial data in $BV$ when $B=\mathbb{I}_n$.

Rest of the paper is organized as follows. In the next subsection we state our main result which deals with  the existence of global solution and uniform BV estimate for \eqref{eqn-parabolic}. In section \ref{sec:uniqueness}, we prove properties of the semigroup solution arising as a limit in Theorem \ref{theorem-1}. We give an out line of the proofs and set up the basics in section \ref{sec:Outline}. We prove BV estimates, $L^1$ stability and the finite speed of propagation in sections \ref{sec:BV}, \ref{sec:stability} and \ref{sec:propagation} respectively. 

%
%
%

   \subsection{Main result}
   We state our first main result which concerns about global existence of smooth solutions to \eqref{eqn-parabolic} for small BV initial data.
   \begin{theorem}\label{theorem-1}
   	Consider the Cauchy problem hyperbolic system with viscosity,
   	\begin{equation}\label{eqn-thm-1}
   		u_t+A(u)u_x=\e(B(u)u_x)_x,\quad u(0,x)=\bar{u}(x).
   	\end{equation}
   We assume that the drift $A$ satisfies \descref{A1}{($\mathcal{H}_A1$.)}, \descref{A2}{($\mathcal{H}_A2$.)} and viscosity matrix $B$ verifies \descref{B}{($\mathcal{H}_B$.)}. Furthermore, we assume that
   \begin{equation}
   	A(u)B(u)=B(u)A(u)\mbox{ for all }u\in \mathcal{U}.
   \end{equation} There exists $L_1,L_2,L_3>0$ and $\de_0>0$ such that the following holds. If $\bar{u}$ satisfies
   \begin{equation}\label{condition-data-thm-1}
   	TV(\bar{u})\leq\de_0\mbox{ and }\lim\limits_{x\rr-\f}\bar{u}(x)\in K,
   \end{equation}
for some compact set $K\subset\mathcal{U}$ then there exists unique solution $u^\e$ to the Cauchy problem \eqref{eqn-thm-1} and it satisfies the following properties
\begin{align}
	TV(u^\e(t))&\leq L_1TV(\bar{u}),\label{thm-1:BV-bound}\\
		\norm{u^\e(t)-v^\e(t)}_{L^1}&\leq L_2\norm{\bar{u}-\bar{v}}_{L^1},\label{thm1-Lipschitz}\\
	\norm{u^\e(t)-u^\e(s)}_{L^1}&\leq L_3\left(\abs{t-s}+\sqrt{\e}\abs{\sqrt{t}-\sqrt{s}}\right),\label{L1-cont}
\end{align} 
where $v^\e$ is the unique solution corresponding to $\bar{v}$ satisfying \eqref{condition-data-thm-1}. 

Furthermore, when $A=Df$ for some $f\in C^1$, as $\e\rr0$ (up to a subsequence), $u^\e(t)\rr u(t,x)$ a solution to hyperbolic system \eqref{eqn:hyperbolic}. 
   \end{theorem}

	Since $A$ is assumed to be strictly hyperbolic, it can be shown that $A$ and $B$ share same set of eigenvectors. Therefore, we can assume that there exists $\{r_i\}$ and $\{l_i\}$ set of right and left eigenvectors of $A$ and $B$ both, that is, $Ar_i=\la_ir_i,\, l_iA=\la_il_i$ and $Br_i=\mu_ir_i,\, l_iB=\mu_il_i$ for $i\leq i\leq n$. Due to this condition, the function $u$ satisfying $u_x=a(\cdot)r_i(u)$ can be a viscous travelling wave for appropriate scalar function $a$ (see Remark \ref{remark:travelling-wave}). Hence we decompose $u_x$ in the basis $\{r_i\}_{i=1}^{n}$. 

Now, from the uniform BV bounds \eqref{thm-1:BV-bound} and the $L^1$ continuity, we can pass to the limit as $\e\rr0$ and the limit solution $u(t)$ forms a Lipschitz semigroup $S_t(\bar{u})$ due to \eqref{thm1-Lipschitz}. In the following section we show that the limit is independent of the subsequence and hence the semigroup is well-defined.

\subsection{Uniqueness of semi-group}\label{sec:uniqueness}
We first recall the Riemann solver for equation of hyperbolic systems (as in \cite{BiB-temple-class})
\begin{equation}\label{eqn:CP-Rie-non}
	u_t+A(u)u_x=0,\mbox{ with }u(0,x)=\left\{\begin{array}{rl}
		u_l&\mbox{ for }x<0,\\
		u_r&\mbox{ for }x>0,
	\end{array}\right.
\end{equation}
with $\abs{u_l-u_r}$ is small enough. We consider the $i$-th rarefaction curve $\si\mapsto \mathcal{R}_{i}(\si;u_-)$ starting from $u_-\in\Omega$ which satisfies
\begin{equation}
	\frac{d}{d\si}\mathcal{R}_i(\si;u_-)=r_i(\mathcal{R}_i(\si;u_-))\mbox{ with }\mathcal{R}_{i}(0;u_-)=u_-.
\end{equation}
By using implicit function theorem and with the help of strict hyperbolicity, there exist $\bar{\la}_1<\cdots<\bar{\la}_{n-1}$, $\{\si_i\}_{i=1}^{n}$ and $\{w_i\}_{i=0}^{n}$ such that
\begin{equation}
	w_0=u_l,\,w_n=u_r\mbox{ and }w_{i}=\mathcal{R}_i(\si;w_{i-1})\mbox{ for }i=1,2,\cdots,n.
\end{equation}
Moreover, $\la_i(\mathcal{R}_i(\theta\si_i;w_{i-1}))\in(\bar{\la}_{i-1},\bar{\la}_i)$ for $\theta\in[0,1]$ and $1\leq i\leq n$ where $\bar{\la}_0:=-\f$ and $\bar{\la}_n:=+\f$. Let us consider scalar flux $F_i$ corresponding to $i$-characteristics defined as follows
\begin{equation}\label{def:F-i}
	F_i(\omega):=\int\limits_{0}^{\omega}\la_i(\mathcal{R}_i(s;w_{i-1}))\,ds.
\end{equation}
Let $z_i$ be the unique entropy solution the following Cauchy problem for scalar conservation laws,
\begin{align}
	z_{i,t}+F_i(z_i)_x&=0,\label{eqn:z_i}\\
	z_i(0,x)&=\left\{\begin{array}{rl}
		0&\mbox{ if }x<0,\\
		\si_i&\mbox{ if }x>0.
	\end{array}\right.
\end{align}
Now, we can describe the solution to \eqref{eqn:CP-Rie-non} as follows,
\begin{equation}\label{soln:Rie}
	u(t,x)=\mathcal{R}_i(z_i(t,x);w_{i-1})\mbox{ for }\frac{x}{t}\in [\bar{\la}_{i-1},\bar{\la}_i]\mbox{ for all }i=1,\cdots,n.
\end{equation}
%

\begin{theorem}\label{theorem-2}
	Let $A$ and $B$ be as in Theorem \ref{theorem-1}. Then for every compact set $K\subset\mathcal{U}$ there exist $L_1,L_2,\de_0$, a closed domain $\mathcal{D}\subset L^1_{loc}(\R)$ and a semigroup $S:[0,\f)\times\mathcal{D}\rr\mathcal{D}$ satisfying the following properties. 
	\begin{description}
		\descitem{($\mathcal{S}1$.)}{S1}  Every function $\bar{u}$ verifying \eqref{condition-data-thm-1} belongs to $\mathcal{D}$. 
		\descitem{($\mathcal{S}2$.)}{S2} For any $\bar{u},\bar{v}\in\mathcal{D}$ with $\bar{u}-\bar{v}\in L^1$,
		\begin{equation}
			\norm{S_{t_1}(\bar{u})-S_{t_2}(\bar{v})}_{L^1}\leq L_1\norm{\bar{u}-\bar{v}}_{L^1}+L_2\abs{t_1-t_2}\mbox{ for any }t_1,t_2\geq0.
		\end{equation}
     	\descitem{($\mathcal{S}3$.)}{S3} For any piece-wise constant initial data $\bar{u}\in\mathcal{D}$ there exists $\tau>0$ such that the following holds. For $t\in [0,\tau]$, $S_t$ coincides with the solution constructed by gluing the Riemann problem solutions (defined by \eqref{soln:Rie}) arising at each jump point. 
			\descitem{($\mathcal{S}4$.)}{S4} For each $\bar{u}\in\mathcal{D}$, $t\mapsto S_t(\bar{u})$ is the unique limit of the sequence $u^{\e_k}(t,\cdot)$ in $L^1_{loc}$ as $k\rr\f$ for any $\e_k\rr0$ where $u^{\e_k}(t,\cdot)$ solves \eqref{eqn-parabolic} with initial data $\bar{u}$.
	\end{description}

\end{theorem}
\begin{remark}
	\begin{enumerate}
	\item We would like to emphasize that the functions contained in $\mathcal{D}$ have small total variation.
	\item When the system is conservative which corresponds to \eqref{eqn:conlaw}, $t\mapsto S_t(\bar{u})$ is the unique weak limit of the system \eqref{eqn:conlaw} in the class of uniqueness described in \cite{Bai-Bre}.
		\item From Theorem \ref{theorem-2}, we can say that the sequence of viscosity solutions $\{u^\e\}$ has a unique limit as $\e\rr0$. We denote $S^B_t(\bar{u})=\lim\limits_{\e\rr0} u^\e(t)$.
		\item Let $S^I_t$ be the semigroup constructed in \cite{BiB-temple-class} for \eqref{eqn:hyperbolic}. Due to the characterization \descref{S1}{($\mathcal{S}1$.)} to \descref{S3}{($\mathcal{S}3$.)} which corresponds to a class of uniqueness of the semi group (see \cite{Bressanbook}), we conclude that the semigroup $S^B_t$ constructed as above coincides with $S^I_t$.
	\end{enumerate}
\end{remark}

	\section{Outline of proof}\label{sec:Outline}
	Our first goal is to establish an existence of global solution $u^\e$ to \eqref{eqn-parabolic} corresponding to an initial data $\bar{u}$. Following \cite{BiB-vv-lim,BiB-temple-class}, we consider the rescaling $t\mapsto t/\e, x\mapsto x/\e$ and the Cauchy problem \eqref{eqn-parabolic}--\eqref{eqn:intial-data} becomes
	\begin{align}
		u_t+A(u)u_x&=(B(u)u_x)_x \quad \quad\mbox{ for }t>0,x\in\R,\label{eqn-main}\\
		u(0,x)&=\bar{u}_0(\e x)\quad \quad\mbox{ for }x\in\R,
	\end{align}
	We denote by $\la_1<\cdots<\la_n$ eigenvalues of $A(u)$. Let us consider $l_1,\cdots,l_n$, $r_1,\cdots,r_n$ as left and right eigenvectors of $A(u)$ such that
	\begin{equation}
		\abs{r_i(u)}=1\mbox{ and }l_i(u)\cdot r_j(u)=\left\{\begin{array}{rl}
			1&\mbox{ if }i=j,\\
			0&\mbox{ if }i\neq j.
		\end{array}\right.
	\end{equation} Following \cite{BiB-temple-class}, we set $u^i_x:=l_i(u)\cdot u_x$ and we have
  \begin{equation}\label{decomp-u_x}
  	u_x=\sum\limits_{i}u^i_xr_i(u).
  \end{equation}
Due to our assumption $\{r_i(u)\}_{i=1}^{n}$ are right eigenvector of $B(u)$ with eigenvalues $\{\mu_i(u)\}_{i=1}^{n}$ respectively. 

\begin{remark}\label{remark:travelling-wave}
	We would like to mention that for general $n\times n$ system Bianchini and Bressan \cite{BiB-vv-lim} used the gradient decomposition in a basis of travelling waves. It can be observed that for a Temple system,  $u(t,x)=U(x-\si t)$ with $u_x(t,x)=a(s)r_i(U(s))$ and $s=x-\sigma t$ forms a travelling wave when $U, a, \si_i$ satisfying the following system of ODE, 
	$$
	\begin{cases}
	\begin{aligned}
		&U^\p(s)=a(s)r_i(U(s)),\\
		 &a^\p(s)=\frac{1}{\mu_i(U(s))}\left[(\la_i(U(s))-\si_i)a(s)-r_i(U(s))\bullet\mu_i\right],\\
		 &\si^\p_i(s)=0.
	\end{aligned}
	\end{cases}
	$$
When $A$ satisfies \descref{A2}{($\mathcal{H}_A2.$)}, the decomposition \eqref{decomp-u_x} is suitable for studying the global existence of smooth solution for \eqref{eqn-parabolic} with $B(u)=\mathbb{I}_n$. On the other hand \eqref{decomp-u_x} fails to provide $L^1$ estimates in time variable $t$ for a system not satisfying \descref{A2}{($\mathcal{H}_A2.$)} (see \cite{BiB-vv-lim}). In our case it is the combination of the assumption \descref{A2}{($\mathcal{H}_A2.$)} and the fact that $A$ and $B$ are commutating which ensure that we can use the decomposition \eqref{decomp-u_x}. Indeed if $A$ and $B$ are not commutating, it is not clear that we can write the travelling wave under the form $u_x=a(s)r_i(U(s))$.
\end{remark}

 From \eqref{eqn-main}, we have
\begin{equation}
	u_t+\sum\limits_{i}\la_iu_x^ir_i=\sum\limits_{i}(u_x^iB(u)r_i)_x=\sum\limits_{i}(\mu_i u_x^ir_i)_x=\sum\limits_{i}(\mu_iu_x^i)_xr_i+\sum\limits_{i,j}\mu_iu^i_xu_x^jr_j\bullet r_i,
\end{equation}
or equivalently,
\begin{equation}\label{eqn-u-t}
		u_t=-\sum\limits_{i}\la_iu_x^ir_i+\sum\limits_{i}\mu_iu_{xx}^ir_i+\sum\limits_{i,j}u_x^iu_x^j(r_j\bullet \mu_i) r_i+\sum\limits_{i,j}\mu_iu^i_xu_x^jr_j\bullet r_i.
\end{equation}
Differentiating \eqref{eqn-u-t} w.r.t $x$ we get
\begin{align*}
	u_{tx}&=-\sum\limits_{i}(\la_iu_x^i)_xr_i-\sum\limits_{i,j}\la_iu_x^iu_x^jr_j\bullet r_i+\sum\limits_{i}(\mu_iu_{xx}^i)_xr_i+\sum\limits_{i,j}\mu_iu_{xx}^iu_x^jr_j\bullet r_i\\
	&+\sum\limits_{i,j}u_{xx}^iu_x^j(r_j\bullet \mu_i) r_i+\sum\limits_{i,j}u_x^iu_{xx}^j(r_j\bullet \mu_i) r_i+\sum\limits_{i,j,k}u_x^iu_x^ju_x^k(r_k\bullet(r_j\bullet \mu_i)) r_i\\
	&+\sum\limits_{i,j,k}u_x^iu_x^ju_x^k(r_j\bullet \mu_i) r_k\bullet r_i+\sum\limits_{i,j,k}u^i_xu_x^ju_x^k(r_k\bullet\mu_i)r_j\bullet r_i+\sum\limits_{i,j}\mu_iu^i_{xx}u_x^jr_j\bullet r_i\\
	&+\sum\limits_{i,j}\mu_iu^i_xu_{xx}^jr_j\bullet r_i+\sum\limits_{i,j,k}\mu_iu^i_xu_x^ju_x^kr_k\bullet (r_j\bullet r_i),
\end{align*}
or equivalently,
\begin{align}
	u_{tx}&=-\sum\limits_{i}(\la_iu_x^i)_xr_i-\sum\limits_{i,j}\la_iu_x^iu_x^jr_j\bullet r_i+\sum\limits_{i}(\mu_iu_{xx}^i)_xr_i+2\sum\limits_{i,j}\mu_iu_{xx}^iu_x^jr_j\bullet r_i\nonumber\\
	&+\sum\limits_{i,j}(u_{xx}^iu_x^j+u_{x}^iu_{xx}^j)(r_j\bullet \mu_i) r_i+\sum\limits_{i,j,k}u_x^iu_x^ju_x^k(r_k\bullet(r_j\bullet \mu_i)) r_i \nonumber\\
	&+2\sum\limits_{i,j,k}u_x^iu_x^ju_x^k(r_j\bullet \mu_i) r_k\bullet r_i+\sum\limits_{i,j}\mu_iu^i_xu_{xx}^jr_j\bullet r_i+\sum\limits_{i,j,k}\mu_iu^i_xu_x^ju_x^kr_k\bullet (r_j\bullet r_i).\label{eqn-u-tx-1}
\end{align}
Differentiating \eqref{decomp-u_x} w.r.t. $t$ we have
\begin{equation*}
	u_{tx}=\sum\limits_{i}u^i_{xt}r_i+\sum\limits_{i}u^i_xu_t\bullet r_i.
\end{equation*}
Using \eqref{eqn-u-t}, we obtain
\begin{align}
		u_{tx}&=\sum\limits_{i}u^i_{xt}r_i-\sum\limits_{i,j}\la_ju_x^ju^i_xr_j\bullet r_i+\sum\limits_{i}u^i_x\mu_ju_{xx}^jr_j\bullet r_i   \nonumber\\
		&+\sum\limits_{i,j,k}u^i_xu_x^ju_x^k(r_k\bullet \mu_j) r_j\bullet r_i+\sum\limits_{i,j,k}\mu_j u^i_xu^j_xu_x^k(r_k\bullet r_j)\bullet r_i. \label{eqn-u-tx-2}
\end{align}
Combining \eqref{eqn-u-tx-1} and \eqref{eqn-u-tx-2} we get
\begin{align*}
	&\sum\limits_{i}(u^i_{xt}+(\la_iu_x^i)_x-(\mu_iu_{xx}^i)_x)r_i\\
	&=-\sum\limits_{i,j}\la_iu_x^iu_x^j(r_j\bullet r_i-r_{i}\bullet r_j)+2\sum\limits_{i,j}\mu_iu_{xx}^iu_x^jr_j\bullet r_i\\
	&+\sum\limits_{i,j}(u_{xx}^iu_x^j+u_{x}^iu_{xx}^j)(r_j\bullet \mu_i) r_i+\sum\limits_{i,j,k}u_x^iu_x^ju_x^k(r_k\bullet(r_j\bullet \mu_i)) r_i\\
	&+2\sum\limits_{i,j,k}u_x^iu_x^ju_x^k(r_j\bullet \mu_i) r_k\bullet r_i+\sum\limits_{i,j}\mu_iu^i_xu_{xx}^jr_j\bullet r_i+\sum\limits_{i,j,k}\mu_iu^i_xu_x^ju_x^kr_k\bullet (r_j\bullet r_i)\\
	&-\sum\limits_{i}u^i_x\mu_ju_{xx}^jr_j\bullet r_i-\sum\limits_{i,j,k}u^i_xu_x^ju_x^k(r_k\bullet \mu_j) r_j\bullet r_i-\sum\limits_{i,j,k}\mu_j u^i_xu^j_xu_x^k(r_k\bullet r_j)\bullet r_i.
\end{align*}
or equivalently,
\begin{align*}
	&\sum\limits_{i}(u^i_{xt}+(\la_iu_x^i)_x-(\mu_iu_{xx}^i)_x)r_i-\sum\limits_{i,j}(u_{xx}^iu_x^j+u_{x}^iu_{xx}^j)(r_j\bullet \mu_i) r_i-\sum\limits_{i,j,k}u_x^iu_x^ju_x^k(r_k\bullet(r_j\bullet \mu_i)) r_i\\
	&=-\sum\limits_{i,j}\la_iu_x^iu_x^j(r_j\bullet r_i-r_{i}\bullet r_j)+2\sum\limits_{i,j}\mu_iu_{xx}^iu_x^jr_j\bullet r_i+\sum\limits_{i,j}(\mu_i-\mu_j)u^i_xu_{xx}^jr_j\bullet r_i\\
	&+2\sum\limits_{i,j,k}u_x^iu_x^ju_x^k(r_j\bullet \mu_i) r_k\bullet r_i+\sum\limits_{i,j,k}\mu_iu^i_xu_x^ju_x^k(r_k\bullet (r_j\bullet r_i)-(r_k\bullet r_i)\bullet r_j)\\
	&-\sum\limits_{i,j,k}u^i_xu_x^ju_x^k(r_k\bullet \mu_j) r_j\bullet r_i.
\end{align*}
Note that
\begin{equation*}
	\sum\limits_{i,j}(u_{xx}^iu_x^j+u_{x}^iu_{xx}^j)(r_j\bullet \mu_i) r_i+\sum\limits_{i,j,k}u_x^iu_x^ju_x^k(r_k\bullet(r_j\bullet \mu_i)) r_i=\sum\limits_{i}(u_x\bullet \mu_i u^i_x)_xr_i.
\end{equation*}
Subsequently, we have
\begin{align*}
	&\sum\limits_{i}(u^i_{xt}+(\la_iu_x^i)_x-(\mu_iu_{x}^i)_{xx})r_i\\
	&=\sum\limits_{i}(u^i_{xt}+(\la_iu_x^i)_x-(\mu_iu_{xx}^i)_x)r_i-\sum\limits_{i}(u_x\bullet \mu_i u^i_x)_xr_i\\
	&=-\sum\limits_{i,j}\la_iu_x^iu_x^j(r_j\bullet r_i-r_{i}\bullet r_j)+2\sum\limits_{i,j}\mu_iu_{xx}^iu_x^jr_j\bullet r_i\\
	&+\sum\limits_{i,j}\mu_iu^i_xu_{xx}^j(r_j\bullet r_i-r_i\bullet r_j)+2\sum\limits_{i,j,k}u_x^iu_x^ju_x^k(r_j\bullet \mu_i) r_k\bullet r_i\\
	&+\sum\limits_{i,j,k}\mu_iu^i_xu_x^ju_x^k(r_k\bullet (r_j\bullet r_i)-(r_k\bullet r_i)\bullet r_j)-\sum\limits_{i,j,k}u^i_xu_x^ju_x^k(r_k\bullet \mu_j) r_j\bullet r_i.
\end{align*}
Hence, we can write
\begin{equation*}
		\sum\limits_{i}(u^i_{xt}+(\la_iu_x^i)_x-(\mu_iu_{x}^i)_{xx})r_i=\sum\limits_{i,j}p_{ij}u_x^iu_x^j+\sum\limits_{i,j}q_{ij}u_{xx}^iu_x^j+\sum\limits_{i,j,k}s_{ijk}u_x^iu_x^ju_x^k,
\end{equation*}
where $p_{ij},q_{ij}$ and $s_{ijk}$ are defined as follows,
\begin{align*}
	p_{ij}&=-\la_i(r_j\bullet r_i-r_{i}\bullet r_j),\\
	q_{ij}&=2\mu_ir_j\bullet r_i+(\mu_i-\mu_j)r_j\bullet r_i,\\
	s_{ijk}&=2(r_j\bullet \mu_i) r_k\bullet r_i+\mu_i(r_k\bullet (r_j\bullet r_i)-(r_k\bullet r_i)\bullet r_j)-(r_k\bullet \mu_j) r_j\bullet r_i.
\end{align*}
Furthermore, we set $p_{jk}^i:=l_i\cdot p_{jk},\,q_{jk}^i:=l_i\cdot q_{jk}$ and $s^i_{jkl}:=l_i\cdot s_{jkl}$. Writing $v_i=u_x^i$, we get
\begin{align*}
	v_{i,t}+(\la_iv_i)_x-(\mu_iv_i)_{xx}&=\sum\limits_{j,k}p_{jk}^iv_jv_k+\sum\limits_{j,k}q_{jk}^iv_{j,x}v_{k}+\sum\limits_{j,k,l}s_{jkl}^iv_jv_kv_l\\
	&=:\phi_i(u,v_1,\cdots,v_n).
\end{align*}
We note that $p^i_{kk}=q^i_{kk}=s_{kkk}^i=0$ for all $i,k$ due to the assumption $r_k\bullet r_k=0$ for all $k$.

To prove the $L^1$ stability we would like to derive an equation for $h$ where $h$ is a perturbation of solution. From \eqref{eqn-main} we have
\begin{align*}
	&u_t+\e h_t+A(u+\e h)u_x+\e A(u+\e h)h_x\\
	&=B(u+\e h)u_{xx}+\e B(u+\e h)h_{xx}+u_x\bullet B(u+\e h)u_x+\e u_x\bullet B(u+\e h)h_x\\
	&+\e h_x\bullet B(u+\e h)u_x+\e^2h_x\bullet B(u+\e h)h_x.
\end{align*}
Using $u_t=-A(u)u_x+B(u)u_{xx}+u_x\bullet B(u) u_x$, we get
\begin{align*}
	&\e h_t+(A(u+\e h)-A(u))u_x+\e A(u+\e h)h_x\\
	&=(B(u+\e h)-B(u))u_{xx}+\e B(u+\e h)h_{xx}+\e u_x\bullet B(u+\e h)h_x\\
	&+u_x\bullet (B(u+\e h)-B(u))u_x+\e h_x\bullet B(u+\e h)u_x+\e^2h_x\bullet B(u+\e h)h_x.
\end{align*}
Dividing by $\e$ and then passing to the limit as $\e\rr0$ we have
\begin{align}
	& h_t+h\bullet A(u)u_x+ A(u)h_x \nonumber\\
	&=B(u)h_{xx}+h\bullet B(u)u_{xx}+(u_x\otimes h):D^2B(u)u_x   \nonumber\\
	&+u_x\bullet B(u)h_x+h_x\bullet B(u)u_x.\label{eqn:h}
\end{align}	
Next, we would like to discuss the way we are going to implement \eqref{eqn:h} in the context of $L^1$ stability. Let $\bar{u}^\theta$ be the initial defined as follows 
\begin{equation}
	\bar{u}^\theta:=\theta \bar{u}+(1-\theta)\bar{v}\mbox{ for some }\bar{u},\bar{v}\in\mathcal{D}.
\end{equation}
Let $u^\theta$ be the solution to \eqref{eqn-main}. Then taking derivative w.r.t $\theta$ we obtain
\begin{align*}
	0&=\frac{\pa}{\pa\theta}[u^\theta_t+A(u^\theta)u^\theta_x-(B(u^\theta)u^\theta_x)_x]\\
	&=\frac{\pa u^\theta_t}{\pa\theta}+\left(\frac{\pa u^\theta}{\pa \theta}\right)\bullet A(u^\theta)u^\theta_x+A(u^\theta)\frac{\pa u^\theta_x}{\pa \theta}+\left(\left(\frac{\pa u^\theta}{\pa \theta}\right)\bullet B(u^\theta)u^\theta_x+B(u^\theta)\frac{\pa u^\theta_x}{\pa \theta}\right)_x.
\end{align*}
Set $h=\frac{\pa u^\theta}{\pa \theta}$ and we can simplify as follows.
\begin{align*}
	0&=h_t+h\bullet A(u^\theta)u^\theta_x+A(u^\theta)h_x-(h\bullet B(u^\theta)u^\theta_x+B(u^\theta)h_x)_x\\
	&=h_t+h\bullet A(u^\theta)u^\theta_x+A(u^\theta)h_x-h_x\bullet B(u^\theta)u^\theta_x-h\bullet B(u^\theta)u^\theta_{xx}\\
	&-(u^\theta_x\otimes h):D^2B(u^\theta)u^\theta_x-u^\theta_x\bullet B(u^\theta)h_x-B(u^\theta)h_{xx}.
\end{align*}
We want to decompose $h$ as $h=\sum\limits_{i}h_ir_i$. Then we would like to get a forced equation for $h_i$. To this end, we first observe that
\begin{align*}
	(B(u)h)_{xx}&=(u_x\bullet B(u)h+B(u)h_x)_x\\
	&=u_{xx}\bullet B(u)h+2u_x\bullet B(u)h_x+(u_x\otimes u_x):D^2B(u)h+B(u)h_{xx}.
\end{align*}
This implies
\begin{align*}
	h_t+ (A(u)h)_x&=(B(u)h)_{xx}-u_{xx}\bullet B(u)h-2u_x\bullet B(u)h_x-(u_x\otimes u_x):D^2B(u)h\\
	&+u_x\bullet A(u)h-h\bullet A(u)u_x+h\bullet B(u)u_{xx}+(u_x\otimes h):D^2B(u)u_x\\
	&+u_x\bullet B(u)h_x+h_x\bullet B(u)u_x\\
	&=(B(u)h)_{xx}+h_x\bullet B(u)u_x-u_{xx}\bullet B(u)h+h\bullet B(u)u_{xx}-u_x\bullet B(u)h_x\\
	&+(u_x\otimes h):D^2B(u)u_x-(u_x\otimes u_x):D^2B(u)h\\
	&+u_x\bullet A(u)h-h\bullet A(u)u_x.
\end{align*}
Recall that $u_x=\sum\limits_{i}v_ir_i$ and $u_{xx}=\sum\limits_{i}v_{i,x}r_i+\sum\limits_{i,j}v_iv_jr_j\bullet r_i$. From decomposition of $h$, we have $h_x=\sum\limits_{i}h_{i,x}r_i+\sum\limits_{i,j}h_{i}v_jr_j\bullet r_i$. Hence, we can obtain
\begin{align*}
	&\sum\limits_{i}h_{i,t}r_i+\sum\limits_{i}h_iu_t\bullet r_i+\sum\limits_{i}(\la_ih_ir_i)_x\\
	&=\sum\limits_{i}(\mu_ih_ir_i)_{xx}+\sum\limits_{i,j}h_{j,x}v_ir_j\bullet B r_i+\sum\limits_{i,j}v_{i}v_{k}h_j(r_k\bullet r_j)\bullet B r_i\\
	&-\sum\limits_{i,j}v_{j,x}h_ir_j\bullet B r_i-\sum\limits_{i,j}v_{j}v_{k}h_i(r_k\bullet r_j)\bullet B r_i+\sum\limits_{i,j}h_{j}v_{i,x}r_j\bullet Br_i\\
	&+\sum\limits_{i,j,k}h_{j}v_{i}v_kr_j\bullet B(r_k\bullet r_i)-\sum\limits_{i,j}h_{i,x}v_{j}r_j\bullet Br_i-\sum\limits_{i,j,k}h_{i}v_{j}v_kr_j\bullet B(r_k\bullet r_i)\\
	&+\sum\limits_{i,j,k}v_{i}v_jh_k(r_j\otimes r_k):D^2B r_i-\sum\limits_{i,j,k}h_{i}v_jv_k(r_j\otimes r_k):D^2B r_i\\
	&+\sum\limits_{i,j}v_jh_ir_j\bullet Ar_i-\sum\limits_{i,j}h_jv_ir_j\bullet Ar_i.
\end{align*}
From \eqref{eqn-u-t} we have
\begin{align*}
	\sum\limits_{i}h_iu_t\bullet r_i&=-\sum\limits_{i,j}\la_jv_jh_ir_j\bullet r_i+\sum\limits_{i,j}\mu_jv_{j,x}h_ir_j\bullet r_i+\sum\limits_{i,j,k}v_jv_kh_i(r_k\bullet \mu_j) r_j\bullet r_i\\
	&+\sum\limits_{i,j,k}\mu_jv_jv_kh_i(r_k\bullet r_j)\bullet r_i.
\end{align*}
Furthermore, we get
\begin{align*}
	\sum\limits_{i}(\mu_ih_ir_i)_{xx}&=\sum\limits_{i}((\mu_ih_i)_xr_i+\mu_ih_iu_x\bullet r_i)_x\\
	&=\sum\limits_{i}(\mu_ih_i)_{xx}r_i+2\sum\limits_{i}(\mu_ih_i)_x(u_x\bullet r_i)+\sum\limits_{i,j}\mu_ih_i((v_jr_j)\bullet r_i)_x\\
	&=\sum\limits_{i}(\mu_ih_i)_{xx}r_i+2\sum\limits_{i,j}(u_x\bullet \mu_i h_i+\mu_i h_{i,x})v_j(r_j\bullet r_i)\\
	&+\sum\limits_{i,j}\mu_ih_iv_{j,x}r_j\bullet r_i+\sum\limits_{i,j,k}\mu_ih_iv_jv_k(r_k\bullet (r_j\bullet r_i))\\
	&=\sum\limits_{i}(\mu_ih_i)_{xx}r_i+2\sum\limits_{i,j,k} h_iv_jv_k (r_k\bullet \mu_i)r_j\bullet r_i+2\sum\limits_{i,j}\mu_i h_{i,x}v_j(r_j\bullet r_i)\\
	&+\sum\limits_{i,j}\mu_ih_iv_{j,x}r_j\bullet r_i+\sum\limits_{i,j,k}\mu_ih_iv_jv_k(r_k\bullet (r_j\bullet r_i)).
\end{align*}
We can now write
\begin{align*}
	&\sum\limits_{i}h_{i,t}r_i-\sum\limits_{i,j}\la_jv_jh_ir_j\bullet r_i+\sum\limits_{i,j}\mu_jv_{j,x}h_ir_j\bullet r_i+\sum\limits_{i,j,k}v_jv_kh_i(r_k\bullet \mu_j) r_j\bullet r_i\\
	&+\sum\limits_{i,j,k}\mu_jv_jv_kh_i(r_k\bullet r_j)\bullet r_i+\sum\limits_{i}(\la_ih_i)_xr_i+\sum\limits_{i,j}\la_ih_iv_jr_j\bullet r_i\\
	&=\sum\limits_{i}(\mu_ih_i)_{xx}r_i+2\sum\limits_{i,j,k} h_iv_jv_k (r_k\bullet \mu_i)r_j\bullet r_i+2\sum\limits_{i,j}\mu_i h_{i,x}v_j(r_j\bullet r_i)\\
	&+\sum\limits_{i,j}\mu_ih_iv_{j,x}r_j\bullet r_i+\sum\limits_{i,j,k}\mu_ih_iv_jv_k(r_k\bullet (r_j\bullet r_i))\\
	&+\sum\limits_{i,j}h_{j,x}v_ir_j\bullet B r_i+\sum\limits_{i,j}v_{i}v_{k}h_j(r_k\bullet r_j)\bullet B r_i\\
	&-\sum\limits_{i,j}v_{j,x}h_ir_j\bullet B r_i-\sum\limits_{i,j}v_{j}v_{k}h_i(r_k\bullet r_j)\bullet B r_i+\sum\limits_{i,j}h_{j}v_{i,x}r_j\bullet Br_i\\
	&+\sum\limits_{i,j,k}h_{j}v_{i}v_kr_j\bullet B(r_k\bullet r_i)-\sum\limits_{i,j}h_{i,x}v_{j}r_j\bullet Br_i-\sum\limits_{i,j,k}h_{i}v_{j}v_kr_j\bullet B(r_k\bullet r_i)\\
	&+\sum\limits_{i,j,k}v_{i}v_jh_k(r_j\otimes r_k):D^2B r_i-\sum\limits_{i,j,k}h_{i}v_jv_k(r_j\otimes r_k):D^2B r_i\\
	&+\sum\limits_{i,j}v_jh_ir_j\bullet Ar_i-\sum\limits_{i,j}h_jv_ir_j\bullet Ar_i,
\end{align*}
or equivalently,
\begin{align*}
	&\sum\limits_{i}(h_{i,t}+(\la_ih_i)_x-(\mu_ih_i)_{xx})r_i\\
	&=\sum\limits_{i,j}v_jh_i(\la_j-\la_i)r_j\bullet r_i-\sum\limits_{i,j,k}v_jv_kh_i(r_k\bullet \mu_j) r_j\bullet r_i-\sum\limits_{i,j,k}\mu_jv_jv_kh_i(r_k\bullet r_j)\bullet r_i\\
	&+2\sum\limits_{i,j,k} h_iv_jv_k (r_k\bullet \mu_i)r_j\bullet r_i+\sum\limits_{i,j,k}\mu_ih_iv_jv_k(r_k\bullet (r_j\bullet r_i))\\
	&+\sum\limits_{i,j}v_{i}v_{k}h_j(r_k\bullet r_j)\bullet B r_i-\sum\limits_{i,j,k}v_{j}v_{k}h_i(r_k\bullet r_j)\bullet B r_i+\sum\limits_{i,j,k}h_{j}v_{i}v_kr_j\bullet B(r_k\bullet r_i)\\
	&-\sum\limits_{i,j,k}h_{i}v_{j}v_kr_j\bullet B(r_k\bullet r_i)+\sum\limits_{i,j,k}v_{i}v_jh_k(r_j\otimes r_k):D^2B r_i-\sum\limits_{i,j,k}h_{i}v_jv_k(r_j\otimes r_k):D^2B r_i\\
	&+\sum\limits_{i,j}(\mu_i-\mu_j)v_{j,x}h_ir_j\bullet r_i-\sum\limits_{i,j}v_{j,x}h_ir_j\bullet B r_i+\sum\limits_{i,j}h_{j}v_{i,x}r_j\bullet Br_i\\
	&+2\sum\limits_{i,j}\mu_i h_{i,x}v_j(r_j\bullet r_i)+\sum\limits_{i,j}h_{j,x}v_ir_j\bullet B r_i-\sum\limits_{i,j}h_{i,x}v_{j}r_j\bullet Br_i\\
	&+\sum\limits_{i,j}v_jh_ir_j\bullet Ar_i-\sum\limits_{i,j}h_jv_ir_j\bullet Ar_i.
\end{align*}
We can now write
\begin{align*}
	&\sum\limits_{i}(h_{i,t}+(\la_ih_i)_x-(\mu_ih_i)_{xx})r_i\\
	&=\sum\limits_{i,j}\hat{p}_{ij}h_iv_j+\sum\limits_{i,j,k}\hat{q}_{ijk}h_iv_jv_k+\sum\limits_{i,j}\hat{s}_{ij}h_{i,x}v_j+\sum\limits_{i,j}\hat{w}_{ij}h_iv_{j,x},
\end{align*}
where $\hat{p}_{ij},\hat{q}_{ijk}, \hat{s}_{ij},\hat{w}_{ij}$ are defined as follows
\begin{align*}
	\hat{p}_{ij}&=(\la_j-\la_i)r_j\bullet r_i+r_j\bullet Ar_i-r_i\bullet Ar_j,\\
	\hat{q}_{ijk}&=-(r_k\bullet \mu_j) r_j\bullet r_i-\mu_j(r_k\bullet r_j)\bullet r_i+2(r_k\bullet \mu_i)r_j\bullet r_i+\mu_i(r_k\bullet (r_j\bullet r_i))\\
	&+(r_k\bullet r_i)\bullet B r_j-(r_k\bullet r_j)\bullet B r_i+r_i\bullet B(r_k\bullet r_j)-r_j\bullet B(r_k\bullet r_i)\\
	&+(r_j\otimes r_i):D^2B r_k-(r_j\otimes r_k):D^2B r_i,\\
	 \hat{s}_{ij}&=2\mu_i (r_j\bullet r_i)+r_i\bullet B r_j-r_j\bullet Br_i,\\
	 	 \hat{w}_{ij}&=	(\mu_i-\mu_j)r_j\bullet r_i-r_j\bullet B r_i+r_i\bullet Br_j.
\end{align*}
We observe that $\hat{p}_{kk}=\hat{q}_{kkk}=\hat{s}_{kk}=\hat{w}_{kk}=0$ for all $k$ due to the assumption $r_k\bullet r_k=0$.
\section{Parabolic estimate}
\begin{lemma}\label{lemma:transformation}
	Let $n\geq 1, p\geq1$, $u\in C^3(\R,\R^n)$ and $\mu_i:\R^n\rr[c_0,\f),1\leq i\leq n$ are $C^2$ functions for some $c_0>0$. Consider $\mathcal{T}:L^p(\R,\R^n)\rr L^p(\R,\R^n)$ defined as follows
	\begin{equation}
		\mathcal{T}(f)_i(x)=f_i(X_i(x))\mbox{ where }X_i(x)=\int\limits_{0}^{x}\frac{1}{\sqrt{\mu_i(u(z))}}\,dz.
	\end{equation}
	Then $\mathcal{T}$ is well-defined and it satisfies
	\begin{equation}\label{inequality-T-1}
		\norm{\mathcal{T}(f)}_{L^p}\leq c_0^{-\frac{1}{2p}}\norm{f}_{L^p}\leq c_0^{-\frac{1}{2p}}m^\frac{1}{2p}\norm{\mathcal{T}(f)}_{L^p}\mbox{ where }m=\sup\limits_{1\leq i\leq n}\norm{\mu_i(u(\cdot))}_{L^\f}.
	\end{equation} 
	Furthermore, for $f\in C^2$ we have
	\begin{align}
		&\norm{(\mathcal{T}(f))^\p}_{L^p}\leq m^\frac{p-1}{2p}\norm{f^\p}_{L^p}\leq c_0^{-\frac{p-1}{2p}}m^\frac{p-1}{2p}\norm{\mathcal{T}(f)^\p}_{L^p},	\label{inequality-T-2}\\
		&\norm{(\mathcal{T}(f))^{\p\p}}_{L^p}\leq 2m^\frac{2p-1}{2p}\norm{f^{\p\p}}_{L^p}+ 2m_1c_0^{-\frac{1}{2p}}\norm{u_x}_{L^\f}\norm{f^\p}_{L^p},	\label{inequality-T-3a}\\
		&\norm{f^{\p\p}}_{L^p}\leq 2c_0^{-\frac{1}{2p}}\norm{(\mathcal{T}(f))^{\p\p}}_{L^p}+ 2m_1c_0^{-\frac{p+1}{p}}\norm{u_x}_{L^\f}\norm{(\mathcal{T}(f))^\p}_{L^p},	\label{inequality-T-3b}
	\end{align}
	where $m_1:=\sup\limits_{1\leq i\leq n}\norm{D\mu_i(u(\cdot))}_{L^\f}$.

\end{lemma}
\begin{proof}
	By using the definition of $\mathcal{T}(f)_i$ it follows that
	\begin{equation*}
		\int\limits_{\R}\abs{\mathcal{T}(f)_i(y)}^p\,dy=	\int\limits_{\R}\abs{f(X_i^{-1}(y))}^p\,dy=	\int\limits_{\R}\abs{f(x)}^p\,\frac{1}{\sqrt{\mu_i(u(x))}}\,dx.
	\end{equation*}
	Therefore, the first inequality in \eqref{inequality-T-1} follows. Similarly, we get the second inequality in \eqref{inequality-T-1}. Now, to see \eqref{inequality-T-2}, we first calculate the following derivatives.
	\begin{equation*}
		\frac{d\mathcal{T}(f)_i}{dy}(y)=\sqrt{\mu_i(u(x))}\frac{df_i}{dx}(x)\mbox{ where }y=X_i(x).
	\end{equation*}
	Hence, we have
	\begin{align*}
		\int\limits_{\R}\abs{\frac{d\mathcal{T}(f)_i}{dy}(y)}^p\,dy&=\int\limits_{\R}[\mu_i(u(X_i^{-1}(y)))]^\frac{p}{2}\abs{\frac{df_i}{dx}(X_i^{-1}(y))}^p\,dy\\
		&=	\int\limits_{\R}[\mu_i(u(x))]^\frac{p-1}{2}\abs{\frac{df_i}{dx}(x)}^p\,dx.
	\end{align*}
	Therefore, the first inequality in \eqref{inequality-T-2} follows and similarly, one can achieve the second inequality in \eqref{inequality-T-2}. Now, we can also derive
	\begin{equation}\label{lemma-1-cal-1}
		\frac{d^2\mathcal{T}(f)_i}{dy^2}(y)= \mu_i(u(x))\frac{d^2f_i}{dx^2}(x)+u_x(x)\cdot D\mu_i(u(x))\frac{df_i}{dx}(x)\mbox{ where }y=X_i(x).
	\end{equation}
	By using the inequality $(a+b)^p\leq 2^{p-1}(a^p+b^p)$ for $p\geq 1, a>0,b>0$, we can obtain
	\begin{align*}
		\int\limits_{\R}\abs{\frac{d^2\mathcal{T}(f)_i}{dy^2}(y)}^p\,dy&\leq 2^{p-1}\int\limits_{\R}[\mu_i(u(X_i^{-1}(y)))]^p\abs{\frac{d^2f_i}{dx^2}(X_i^{-1}(y))}^p\,dy\\
		&+2^{p-1}\int\limits_{\R}\abs{u_x(X_i^{-1}(y))\cdot D\mu_i(u(X_i^{-1}(y)))\frac{df_i}{dx}(X_i^{-1}(y))}^p\,dy\\
		&=2^{p-1}\int\limits_{\R}[\mu_i(u(x))]^\frac{2p-1}{2}\abs{\frac{d^2f_i}{dx^2}(x)}^p\,dx\\
		&+2^{p-1}\norm{u_x}^p_{L^\f}\norm{D\mu_i(u)}^p_{L^\f}\int\limits_{\R}[\mu_i(u(x))]^{-\frac{1}{2}}\abs{\frac{df_i}{dx}(x)}^p\,dx.
	\end{align*}
	Hence, the inequality \eqref{inequality-T-3a} follows. Again from \eqref{lemma-1-cal-1} we have
	\begin{align*}
		\int\limits_{\R}\abs{\frac{d^2f_i}{dx^2}(x)}^p\,dx&\leq 2^{p-1}\int\limits_{\R}\frac{1}{\mu_i(u(x))}\abs{\frac{d^2\mathcal{T}(f)_i}{dy^2}(X_i(x))}^p\,dx\\
		&+2^{p-1}\int\limits_{\R}\frac{1}{\mu_i(u(x))}\abs{u_x(x)\cdot D\mu_i(u(x))\frac{df_i}{dx}(x)}^p\,dx\\
		&\leq 2^{p-1}\int\limits_{\R}\frac{1}{\sqrt{\mu_i(u(X_i^{-1}(y)))}}\abs{\frac{d^2\mathcal{T}(f)_i}{dy^2}(y)}^p\,dy\\
		&+2^{p-1}c_0^{-1}\norm{u_x}^p_{L^\f}\norm{D\mu_i(u)}^p_{L^\f}\int\limits_{\R}\abs{\frac{df_i}{dx}(x)}^p\,dx.
	\end{align*}
	Therefore, by using \eqref{inequality-T-2}, we obtain the inequality \eqref{inequality-T-3b}.
\end{proof}
We consider $G$ as the fundamental solution of the following parabolic equation
\begin{equation}
	w_t+A_2^*w_x=w_{xx},
\end{equation}
where $A_2^*=A_1^*B_1^{-1/2}(u^*)$. The function $G$ satisfies the following estimates
\begin{equation}\label{def:kappa}
	\norm{G(t,\cdot)}_{L^1}\leq \kappa,\quad	\norm{G_x(t,\cdot)}_{L^1}\leq \frac{\kappa}{\sqrt{t}},\quad	\norm{G(t,\cdot)}_{L^1}\leq \frac{\kappa}{t},
\end{equation}
for some constant $\kappa>0$.

To have the parabolic estimates for \eqref{eqn-main}, we define the following constants
\begin{align}
	\kappa_1&=2\max\{c_0^{-\frac{1}{2}},m^\frac{1}{2},m_1,1\},\\
	\kappa_A&=\sup\limits_{\abs{u-u^*}\leq \bar{\de}}\max\limits_{k=1,2}\left\{\abs{A(u)},\abs{D^kA(u)},1\right\},\\
	\kappa_B&=\sup\limits_{\abs{u-u^*}\leq \bar{\de}}\max\limits_{k=1,2,3,1\leq i\leq n}\left\{\abs{B(u)},\abs{B^{-1}(u)},\abs{D^{k} B(u)},(\mu_i+\mu_i^{-1}),\abs{D^k\mu_i}\right\},\\
	\kappa_P&=\sup\limits_{\abs{u-u^*}\leq \bar{\de}}\max\limits_{k=1,2,3}\left\{\abs{P(u)},\abs{P^{-1}(u)},\abs{D^{k} P(u)},\abs{D^{k} P^{-1}(u)}\right\}. \label{def:kappa-P}
\end{align}
where $m,m_1,c_0$ are as in Lemma \ref{lemma:transformation}. By the definition we get $\kappa_1,\kappa_A\geq1$. Since $(\mu_i+\mu_i^{-1})\geq1$ and $\abs{P(u)}\abs{P^{-1}(u)}\geq1$, we have $\kappa_B,\kappa_P\geq1$.

\begin{proposition}\label{prop:parabolic}
	Let $u$ be a solution to the equation \eqref{eqn-main} satisfying 
	\begin{equation}\label{assumption:u-L1}
		\norm{u_x(t,\cdot)}_{L^1}\leq \de_0\mbox{ for all }t\in[0,\hat{t}]\mbox{ where }\hat{t}:=\left(\frac{1}{2000\kappa_1^4\kappa_A^3\kappa_B^6\kappa_P^{20}\de_0}\right)^2,
	\end{equation}
	for some $\de_0<1$ and $\kappa,\kappa_1,\kappa_A,\kappa_B,\kappa_P$ are as in \eqref{def:kappa}--\eqref{def:kappa-P}. Then we have
	\begin{equation}\label{estimate:parabolic-1}
		\norm{u_{xx}(t,\cdot)}_{L^1}\leq \frac{2\kappa\kappa_1^2\kappa_P^2\de_0}{\sqrt{t}}.
	\end{equation}
\end{proposition}
\begin{proof} We will argue by contradiction. Before we proceed for that we would like to make a suitable change of variable. To this end, first we can take derivative w.r.t $x$ on both sides of \eqref{eqn-main} to get
	\begin{equation}
		(u_x)_t+A(u)u_{xx}=B(u)u_{xxx}+(u_x\bullet B(u)u_x)_x-u_x\bullet A(u)u_x.
	\end{equation}
	We would like to make a change of variable $v=P(u)u_x$. To this end, we write
	\begin{align*}
		(Pu_x)_t+PA(u)P^{-1}(P(u)u_{x})_x&=PBP^{-1}(Pu_x)_{xx}-PBP^{-1}(u_x\bullet P(u)u_x)_x\\
		&-PBP^{-1}[u_x\bullet Pu_{xx}]+u_t\bullet P(u)u_x\\
		&+PA(u)P^{-1}(u_x\bullet P(u)u_{x})\\
		&+P(u_x\bullet B(u)u_x)_x-Pu_x\bullet A(u)u_x.
	\end{align*}
	Then, if $A_1=PAP^{-1}$ and $B_1=\mbox{diag}(\mu_1(u),\cdots,\mu_n(u))=PBP^{-1}$, we have
	\begin{align*}
		v_t+A_1(u)v_x&=B_1(u)v_{xx}-B_1((P^{-1}(u)v)\bullet PP^{-1}v)_x-B_1[(P^{-1}v)\bullet P(P^{-1}v)_x]\\
		&+u_t\bullet P(u)P^{-1}v+A_1(u)(P^{-1}(u)v\bullet P(u)P^{-1}v)\\
		&+P(P^{-1}v\bullet B(u)P^{-1}v)_x-P(P^{-1}v)\bullet A(u)P^{-1}v.
	\end{align*}
	Hence,
	\begin{align*}
		v_t+A_1(u)v_x&=B_1(u)v_{xx}-B_1((u_x\bullet P^{-1}(u)v)\bullet PP^{-1}v)\\
		&-B_1((P^{-1}(u)v_x)\bullet PP^{-1}v)-B_1(u_x\otimes (P^{-1}(u)v)\bullet D^2PP^{-1}v)\\
		&-2B_1[(P^{-1}v)\bullet P(P^{-1}v)\bullet P^{-1}v]-2B_1[(P^{-1}v)\bullet PP^{-1}v_x]\\
		&+u_t\bullet P(u)P^{-1}v+A_1(u)(P^{-1}(u)v\bullet P(u)P^{-1}v)\\
		&+P(P^{-1}v)\bullet P^{-1}v\bullet B(u)P^{-1}v+PP^{-1}v_x\bullet B(u)P^{-1}v\\
		&+P(P^{-1}v)\otimes(P^{-1}v):D^2B(u)P^{-1}v+P(P^{-1}v)\bullet B(u)(P^{-1}v)\bullet P^{-1}v\\
		&+P(P^{-1}v)\bullet B(u)P^{-1}v_x-P(P^{-1}v)\bullet A(u)P^{-1}v.
	\end{align*}
	We observe that $u_t=BP^{-1}v_x+B[(P^{-1}(u)v)\bullet P^{-1}(u)v]-AP^{-1}v+(P^{-1}(u)v)\bullet B(u)P^{-1}(u)v$. Then we get
	\begin{equation*}
		v_t+A_1^*v_x=B_1(u)v_{xx}+\mathcal{R},
	\end{equation*}
	where $\mathcal{R}=\mathcal{R}(u,v,v_x)$ is defined as follows
	\begin{align}
		\mathcal{R}(u,v,v_x)&:=[A_1^*-A_1(u)]v_x-B_1(((P^{-1}(u)v)\bullet P^{-1}(u)v)\bullet PP^{-1}v)\nonumber\\
		&-B_1((P^{-1}(u)v_x)\bullet PP^{-1}v)-B_1((P^{-1}(u)v)\otimes (P^{-1}(u)v)\bullet D^2PP^{-1}v)\nonumber\\
		&-2B_1[(P^{-1}v)\bullet P(P^{-1}v)\bullet P^{-1}v]-2B_1[(P^{-1}v)\bullet PP^{-1}v_x]\nonumber\\
		&+[BP^{-1}v_x+B[(P^{-1}(u)v)\bullet P^{-1}(u)v]-AP^{-1}v]\bullet P(u)P^{-1}v\nonumber\\
		&+[(P^{-1}(u)v)\bullet B(u)P^{-1}(u)v]\bullet P(u)P^{-1}v+A_1(u)(P^{-1}(u)v\bullet P(u)P^{-1}v)\nonumber\\
		&+P(P^{-1}v)\bullet P^{-1}v\bullet B(u)P^{-1}v+PP^{-1}v_x\bullet B(u)P^{-1}v\nonumber\\
		&+P(P^{-1}v)\otimes(P^{-1}v):D^2B(u)P^{-1}v+P(P^{-1}v)\bullet B(u)(P^{-1}v)\bullet P^{-1}v\nonumber\\
		&+P(P^{-1}v)\bullet B(u)P^{-1}v_x-P(P^{-1}v)\bullet A(u)P^{-1}v.\label{def:remainder-R}
	\end{align}Next, we do a change of variable $v\mapsto \tilde{v}$ such that $v(t,x)=(\tilde{v}_i(t,X_i(t,x)))$ where $ (X_i)_x=\frac{1}{\sqrt{\mu_i(u)}}$.  We observe that
	\begin{equation*}
		v_t=\tilde{v}_t+\sum\limits_{i}\tilde{v}_{i,x}X_{i,t},\quad v_x=B_1^{-1/2}\tilde{v}_{x}\quad\mbox{and }v_{xx}=B_1^{-1}\tilde{v}_{xx}+(P^{-1}v)\bullet B_1^{-1/2}\tilde{v}_{x}.
	\end{equation*}
	Then we have
	\begin{align}
		\tilde{v}_t+A_2^*\tilde{v}_x&=\tilde{v}_{xx}+\mathcal{T}([A_2^*-A^*_1B_1^{-1/2}(u)]\tilde{v}_x+B_1(P^{-1}{v})\bullet B_1^{-1/2}\tilde{v}_{x})\nonumber\\
		&-\sum\limits_{i}\mathcal{T}(\tilde{v}_{i,x}X_{i,t})+ \mathcal{T}(\mathcal{R}),\label{eqn-tilde-v}
	\end{align}
	where $A_2^*=A_1^*B_1^{-1/2}(u^*)$. 
	
	We first prove \eqref{estimate:parabolic-1} for smooth initial data.  We argue by contradiction. To this end, first we assume that the conclusion \eqref{estimate:parabolic-1} does not hold. Due to the assumption of smoothness of initial data, solution is smooth up to a small time and due to the continuity we can assume that there exists a time $t^*$ such that \eqref{estimate:parabolic-1} holds for $t\in[0,t^*]$ and equality attains at $t=t^*$. We compute 
	\begin{align}
		X_{i,t}=-\int\limits_{0}^{x}\frac{u_t\cdot D\mu_i(u)}{2(\mu_i(u))^{3/2}}\,dx=\int\limits_{0}^{x}\frac{(B(u)u_x)_x\cdot D\mu_i(u)-A(u)u_x\cdot D\mu_i(u)}{2(\mu_i(u))^{3/2}}\,dx.
	\end{align}
	Using integration by we can have
	\begin{align*}
		X_{i,t}&=\int\limits_{0}^{x}\frac{(B(u)u_x)_x\cdot D\mu_i(u)-A(u)u_x\cdot D\mu_i(u)}{2(\mu_i(u))^{3/2}}\,dx\\
		&=\int\limits_{0}^{x}\left[-\frac{(B(u)u_x)\otimes u_x: D^2\mu_i(u)}{2(\mu_i(u))^{3/2}}-\frac{3((B(u)u_x)\cdot D\mu_i(u))(u_x\cdot D\mu_i(u))}{4(\mu_i(u))^{3/2}}\right]\,dx\\
		&+\frac{B(u)u_x\cdot D\mu_i(u)}{2(\mu_i(u))^{3/2}}(x)-\frac{B(u)u_x\cdot D\mu_i(u)}{2(\mu_i(u))^{3/2}}(0)-\int\limits_{0}^{x}\frac{A(u)u_x\cdot D\mu_i(u)}{2(\mu_i(u))^{3/2}}\,dx.
	\end{align*}
	Then we get
	\begin{align}
		\norm{X_{i,t}(t,\cdot)}_{L^\f(\R)}&\leq 2\kappa_B^5\kappa_P^2\norm{v}_{L^1}\norm{v}_{L^\f}+4\kappa_B^4\kappa_P\norm{v}_{L^\f}+ \kappa_A\kappa^3_B\kappa_P\norm{v}_{L^1}\nonumber\\
		&\leq [2\kappa_B^5\kappa_P^2\de_0+4\kappa_B^4\kappa_P]\norm{v}_{L^\f}+ \kappa_A\kappa^3_B\kappa_P\de_0\nonumber\\
		&\leq 8\kappa_B^5\kappa_A\kappa_P^2[\norm{v}_{L^\f}+\de_0].
	\end{align}
	Furthermore, we observe that
	\begin{align}
		&\norm{\tilde{v}}_{L^1}\leq \kappa_1\norm{v}_{L^1}\leq \kappa_P\kappa_1\norm{u_x}_{L^1},\\
		&\norm{\tilde{v}_x}_{L^1}\leq \kappa_1\norm{v_x}_{L^1}\leq \kappa_P\kappa_1\norm{u_{xx}}_{L^1},\\
		&\norm{u_{xx}}_{L^1}\leq \kappa_P\norm{v_x}_{L^1}\leq \kappa_P\kappa_1\norm{\tilde{v}_x}_{L^1}.
	\end{align}
	From \eqref{eqn-tilde-v} we can write
	\begin{align*}
		\tilde{v}_{x}&=G_x(t/2)\star \tilde{v}(t/2)+\int\limits_{t/2}^{t}G_x(t-s)\star \Big\{\mathcal{T}([A_2^*-A^*_1B_1^{-1/2}(u)]\tilde{v}_x\\
		&\quad\quad\quad+B_1(P^{-1}{v})\bullet B_1^{-1/2}\tilde{v}_{x})-\sum\limits_{i}\mathcal{T}(\tilde{v}_{i,x}X_{i,t})+ \mathcal{T}(\mathcal{R})\Big\}\,ds.
	\end{align*}
	Using our previous observation, we have
	\begin{align*}
		\norm{\tilde{v}_{x}(t)}_{L^1}&\leq \norm{G_x(t/2)}_{L^1}\norm{\tilde{v}(t/2)}_{L^1}+\int\limits_{t/2}^{t}\norm{G_x(t-s)}_{L^1}\norm{\mathcal{T}([A_2^*-A^*_1B_1^{-1/2}(u)]\tilde{v}_x)}_{L^1}\,ds\\
		&+\int\limits_{t/2}^{t}\norm{G_x(t-s)}_{L^1}\norm{\mathcal{T}(B_1(P^{-1}{v})\bullet B_1^{-1/2}\tilde{v}_{x})-\sum\limits_{i}\mathcal{T}(\tilde{v}_{i,x}X_{i,t})+ \mathcal{T}(\mathcal{R})}_{L^1}\,ds\\
		&\leq \frac{2\kappa\de_0}{\sqrt{t}}+\int\limits_{t/2}^{t}\frac{\kappa\kappa_1}{\sqrt{t-s}}\norm{[A_2^*-A^*_1B_1^{-1/2}(u)]\tilde{v}_x}_{L^1}\,ds\\
		&+\int\limits_{t/2}^{t}\frac{\kappa\kappa_1}{\sqrt{t-s}}\norm{ B_1(P^{-1}{v})\bullet B_1^{-1/2}\tilde{v}_{x}-\sum\limits_{i}\tilde{v}_{i,x}X_{i,t}+\mathcal{R}}_{L^1}\,ds.	
	\end{align*}
	We note that
	\begin{align*}
		\norm{[A_2^*-A^*_1B_1^{-1/2}(u)]\tilde{v}_x}_{L^1}&\leq \kappa_A\kappa_B\norm{u-u^*}_{L^\f}\norm{\tilde{v}_x}_{L^1}\\
		&\leq \kappa_A\kappa_B\kappa_P\kappa_1\norm{u-u^*}_{L^\f}\norm{u_{xx}}_{L^1},\\
		\norm{B_1(P^{-1}{v})\bullet B_1^{-1/2}\tilde{v}_{x}}_{L^1}&\leq \kappa_B^2\kappa_P\norm{v}_{L^\f}\norm{\tilde{v}_x}_{L^1}\\
		&\leq \kappa_B^2\kappa^3_P\kappa_1\norm{u_x}_{L^\f}[\norm{u_{xx}}_{L^1}+\norm{u_x}_{L^1}\norm{u_x}_{L^\f}]\\
		&\leq 2\kappa_B^2\kappa^3_P\kappa_1\norm{u_x}_{L^\f}\norm{u_{xx}}_{L^1},\\
		\norm{-\sum\limits_{i}\tilde{v}_{i,x}X_{i,t}}_{L^1}&\leq \sum\limits_{i}\norm{\tilde{v}_{i,x}}_{L^1}\norm{X_{i,t}}_{L^\f}\\
		&\leq8\kappa_B^5\kappa_A\kappa_P^2\sum\limits_{i}\norm{\tilde{v}_{i,x}}_{L^1}[\norm{v}_{L^\f}+\de_0]\\
		&\leq 8\kappa_B^5\kappa_A\kappa_1\kappa^4_P\norm{u_{xx}}[\norm{u_x}_{L^\f}+\de_0].
	\end{align*}
	From \eqref{def:remainder-R} we get
	\begin{align*}
		\norm{\mathcal{R}}_{L^1}&\leq 30\kappa_B^5\kappa_P^6\left[\norm{v}_{L^\f}^2\norm{v}_{L^1}+\norm{v}_{L^\f}\norm{v_x}_{L^1}+\norm{v}_{L^\f}\norm{v}_{L^1}\right]\\
		&\leq 30\kappa_B^5\kappa_P^8\left[\norm{u_x}_{L^\f}^2\norm{u_x}_{L^1}+\norm{u_x}_{L^\f}\norm{u_{xx}}_{L^1}+\norm{u_x}_{L^\f}\norm{u_x}_{L^1}\right]\\
		&\leq 90\kappa_B^5\kappa_P^8\norm{u_{xx}}^2_{L^1}.
	\end{align*}
	Therefore, we get
	\begin{align*}
		\norm{\tilde{v}_{x}(t)}_{L^1}
		&\leq \frac{\sqrt{2}\kappa\kappa_1\kappa_P\de_0}{\sqrt{t}}+600\kappa_1^6\kappa^3\kappa_A\kappa_B^7\kappa_P^{12}\int\limits_{t/2}^{t}\frac{1}{\sqrt{t-s}}\left[\frac{\de_0^2}{s}+\frac{\de_0^2}{\sqrt{s}}\right]\,ds\\
		&< \frac{2\kappa\kappa_1\kappa_P\de_0}{\sqrt{t}},
	\end{align*}
	which implies
	\begin{equation}
		\norm{u_{xx}(t^*,\cdot)}_{L^1}<\frac{2\kappa\kappa^2_1\kappa^2_P\de_0}{\sqrt{t^*}}.
	\end{equation}
  This contradicts with the assumption that equality holds in \eqref{estimate:parabolic-1} at $t=t^*$. Hence, by density argument we conclude that \eqref{estimate:parabolic-1} holds for BV initial data as well.
    \end{proof}

Next we prove a similar parabolic estimate for $h$ satisfying the linearized equation \eqref{eqn:h}.
\begin{proposition}\label{prop:parabolic-h}
	Let $u$ be as in Proposition \ref{prop:parabolic}. Let $h$ be a solution to the equation \eqref{eqn:h} satisfying 
	\begin{equation}\label{assumption:h-L1}
		\norm{h(t,\cdot)}_{L^1}\leq \de_0\mbox{ for all }t\in[0,\hat{t}]\mbox{ where }\hat{t}:=\left(\frac{1}{2000\kappa_1^4\kappa_A^3\kappa_B^6\kappa_P^{20}\de_0}\right)^2,
	\end{equation}
	for some $\de_0<1$ and $\kappa,\kappa_1,\kappa_A,\kappa_B,\kappa_P$ are as in \eqref{def:kappa}--\eqref{def:kappa-P}. Then we have
	\begin{equation}\label{estimate:parabolic-h}
		\norm{h_{x}(t,\cdot)}_{L^1}\leq \frac{2\kappa\kappa_1^2\kappa_P^2\de_0}{\sqrt{t}}.
	\end{equation}
\end{proposition}	
	
\begin{proof} Similar to the proof of Proposition \ref{prop:parabolic}, we argue by contradiction. First, we make a change of variable. We multiply \eqref{eqn:h} by $P$ and arrange the terms to obtain
	\begin{align*}
		(Ph)_t+PA(u)P^{-1}(Ph)_x&=PBP^{-1}(Ph)_{xx}+P(h\bullet B(u)u_{xx})+P(u_x\otimes h):D^2B(u)u_x\\
		&+u_x\bullet B(u)h_x+h_x\bullet B(u)u_x+u_t\bullet Ph+PA(u)P^{-1}(u_x\bullet Ph)\\
		&-P(h\bullet A(u)u_x).
	\end{align*}
We make a change of variable $g=P(u)h$. Then, if $A_1=PAP^{-1}$ and $B_1=\mbox{diag}(\mu_1(u),\cdots,\mu_n(u))=PBP^{-1}$, we have
\begin{align*}
	g_t+A_1g_x&=B_1g_{xx}+P(P^{-1}g\bullet B(u)u_{xx})+P(u_x\otimes (P^{-1}g)):D^2B(u)u_x\\
	&+u_x\bullet B(u)(u_x\bullet P^{-1}g+P^{-1}g_x)+(u_x\bullet P^{-1}g+P^{-1}g_x)\bullet B(u)u_x\\
	&+u_t\bullet P(P^{-1}g)+PA(u)P^{-1}(u_x\bullet P(P^{-1}g))-P((P^{-1}g)\bullet A(u)u_x).
\end{align*}
By applying \eqref{eqn-main} we get
\begin{align*}
	g_t+A_1g_x&=B_1g_{xx}+P(P^{-1}g\bullet B(u)u_{xx})+P(u_x\otimes (P^{-1}g)):D^2B(u)u_x\\
	&+u_x\bullet B(u)(u_x\bullet P^{-1}g+P^{-1}g_x)+(u_x\bullet P^{-1}g+P^{-1}g_x)\bullet B(u)u_x\\
	&+(Bu_{xx}+u_x\bullet Bu_x-Au_x)\bullet P(P^{-1}g)+PA(u)P^{-1}(u_x\bullet P(P^{-1}g))\\
	&-P((P^{-1}g)\bullet A(u)u_x).
\end{align*}
We write
	\begin{equation*}
		g_t+A_1^*g_x=B_1(u)g_{xx}+\mathcal{R}_1,
	\end{equation*}
	where $\mathcal{R}_1=\mathcal{R}_1(u,u_x,u_{xx},g)$ is defined as follows
	\begin{align}
		\mathcal{R}_1(u,u_x,u_{xx},g)&:=[A_1^*-A_1(u)]g_x+P(P^{-1}g\bullet B(u)u_{xx})\nonumber\\
		&+P(u_x\otimes (P^{-1}g)):D^2B(u)u_x+u_x\bullet B(u)(u_x\bullet P^{-1}g+P^{-1}g_x)\nonumber\\
		&+(u_x\bullet P^{-1}g+P^{-1}g_x)\bullet B(u)u_x+(Bu_{xx}+u_x\bullet Bu_x-Au_x)\bullet P(P^{-1}g)\nonumber\\
		&+PA(u)P^{-1}(u_x\bullet P(P^{-1}g))-P((P^{-1}g)\bullet A(u)u_x).\label{def:remainder-R-1}
	\end{align}
Next, similar to Proposition \ref{prop:parabolic} we do a change of variable $g\mapsto \tilde{g}$ such that $g(t,x)=(\tilde{g}_i(t,X_i(t,x)))$ where $ (X_i)_x=\frac{1}{\sqrt{\mu_i(u)}}$.  We observe that
	\begin{equation*}
		g_t=\tilde{g}_t+\sum\limits_{i}\tilde{g}_{i,x}X_{i,t},\quad g_x=B_1^{-1/2}\tilde{g}_{x}\quad\mbox{and }g_{xx}=B_1^{-1}\tilde{g}_{xx}+u_x\bullet B_1^{-1/2}\tilde{g}_{x}.
	\end{equation*}
	Then we have
	\begin{align}
		\tilde{g}_t+A_2^*\tilde{g}_x&=\tilde{g}_{xx}+\mathcal{T}([A_2^*-A^*_1B_1^{-1/2}(u)]\tilde{g}_x+B_1(u_x\bullet B_1^{-1/2}\tilde{g}_{x}))\nonumber\\
		&-\sum\limits_{i}\mathcal{T}(\tilde{g}_{i,x}X_{i,t})+ \mathcal{T}(\mathcal{R}_1),\label{eqn-tilde-g}
	\end{align}
	where $A_2^*=A_1^*B_1^{-1/2}(u^*)$. 
	
	We first prove \eqref{estimate:parabolic-h} for smooth initial data.  We argue by contradiction. To this end, first we assume that the conclusion \eqref{estimate:parabolic-h} does not hold. Due to the assumption of smoothness of initial data, solution is smooth up to a small time and due to the continuity we can assume that there exists a time $t^*$ such that \eqref{estimate:parabolic-h} holds for $t\in[0,t^*]$ and equality attains at $t=t^*$. From Proposition \ref{prop:parabolic} we recall
	\begin{equation*}
		\norm{X_{i,t}(t,\cdot)}_{L^\f(\R)}\leq 6\kappa_B^5\kappa_A[\norm{u_x}_{L^\f}+\de_0]\mbox{ for }t\in[0,t^*].
	\end{equation*}
	From \eqref{eqn-tilde-g} we can write
	\begin{align*}
		\tilde{g}_{x}&=G_x(t/2)\star \tilde{g}(t/2)+\int\limits_{t/2}^{t}G_x(t-s)\star \Big\{\mathcal{T}([A_2^*-A^*_1B_1^{-1/2}(u)]\tilde{g}_x\\
		&\quad\quad\quad+B_1(u_x\bullet B_1^{-1/2}\tilde{g}_{x}))-\sum\limits_{i}\mathcal{T}(\tilde{g}_{i,x}X_{i,t})+ \mathcal{T}(\mathcal{R}_1)\Big\}\,ds.
	\end{align*}
	Using our previous observation, we have
	\begin{align*}
		\norm{\tilde{g}_{x}(t)}_{L^1}&\leq \norm{G_x(t/2)}_{L^1}\norm{\tilde{g}(t/2)}_{L^1}+\int\limits_{t/2}^{t}\norm{G_x(t-s)}_{L^1}\norm{\mathcal{T}([A_2^*-A^*_1B_1^{-1/2}(u)]\tilde{g}_x)}_{L^1}\,ds\\
		&+\int\limits_{t/2}^{t}\norm{G_x(t-s)}_{L^1}\norm{\mathcal{T}(B_1(u_x\bullet B_1^{-1/2}\tilde{g}_{x}))-\sum\limits_{i}\mathcal{T}(\tilde{g}_{i,x}X_{i,t})+ \mathcal{T}(\mathcal{R}_1)}_{L^1}\,ds\\
		&\leq \frac{2\kappa\de_0}{\sqrt{t}}+\int\limits_{t/2}^{t}\frac{\kappa\kappa_1}{\sqrt{t-s}}\norm{[A_2^*-A^*_1B_1^{-1/2}(u)]\tilde{g}_x}_{L^1}\,ds\\
		&+\int\limits_{t/2}^{t}\frac{\kappa\kappa_1}{\sqrt{t-s}}\norm{ B_1(u_x\bullet B_1^{-1/2}\tilde{g}_{x})-\sum\limits_{i}\tilde{g}_{i,x}X_{i,t}+\mathcal{R}_1}_{L^1}\,ds.	
	\end{align*}
	We note that
	\begin{align*}
		\norm{[A_2^*-A^*_1B_1^{-1/2}(u)]\tilde{g}_x}_{L^1}&\leq \kappa_A\kappa_B\norm{u-u^*}_{L^\f}\norm{\tilde{g}_x}_{L^1}\\
		&\leq \kappa_A\kappa_B\kappa_P\kappa_1\norm{u-u^*}_{L^\f}\norm{h_{x}}_{L^1},\\
		\norm{B_1(u_x\bullet B_1^{-1/2})\tilde{g}_{x}}_{L^1}&\leq \kappa_B^2\kappa_P\norm{u_x}_{L^\f}\norm{\tilde{g}_x}_{L^1}\\
		&\leq \kappa_B^2\kappa^3_P\kappa_1\norm{u_x}_{L^\f}[\norm{h_{x}}_{L^1}+\norm{u_x}_{L^1}\norm{h}_{L^\f}]\\
		&\leq 2\kappa_B^2\kappa^3_P\kappa_1(\norm{u_x}_{L^\f}+\de_0)\norm{h_{x}}_{L^1},\\
		\norm{-\sum\limits_{i}\tilde{g}_{i,x}X_{i,t}}_{L^1}&\leq \sum\limits_{i}\norm{\tilde{g}_{i,x}}_{L^1}\norm{X_{i,t}}_{L^\f}\\
		&\leq6\kappa_B^5\kappa_A\sum\limits_{i}\norm{\tilde{g}_{i,x}}_{L^1}[\norm{u_x}_{L^\f}+\de_0]\\
		&\leq 6\kappa_B^5\kappa_A\kappa_1\kappa^2_P\norm{h_{x}}[\norm{u_x}_{L^\f}+\de_0].
	\end{align*}
	From \eqref{def:remainder-R-1} we get
	\begin{align*}
		\norm{\mathcal{R}_1}_{L^1}&\leq 30\kappa_B^5\kappa_P^6\left[\norm{u_x}_{L^1}\norm{g_x}_{L^1}+\norm{u_{xx}}^2_{L^1}\norm{g}_{L^1}+\norm{u_{xx}}_{L^1}(\norm{g_x}_{L^1}+\norm{g}_{L^1})\right]\\
		&\leq 30\kappa_B^5\kappa_P^8\left[\norm{u_x}_{L^1}\norm{h_x}_{L^1}+\norm{u_{xx}}^2_{L^1}\norm{h}_{L^1}+\norm{u_{xx}}_{L^1}(\norm{h_x}_{L^1}+\norm{h}_{L^1})\right].
	\end{align*}
	Therefore, we get
	\begin{align*}
		\norm{\tilde{g}_{x}(t)}_{L^1}
		&\leq \frac{\sqrt{2}\kappa\kappa_1\kappa_P\de_0}{\sqrt{t}}+600\kappa_1^6\kappa^3\kappa_A\kappa_B^7\kappa_P^{12}\int\limits_{t/2}^{t}\frac{1}{\sqrt{t-s}}\left[\frac{\de_0^2}{s}+\frac{\de_0^2}{\sqrt{s}}\right]\,ds\\
		&< \frac{2\kappa\kappa_1\kappa_P\de_0}{\sqrt{t}},
	\end{align*}
	which implies
	\begin{equation}
		\norm{h_{x}(t,\cdot)}_{L^1}<\frac{2\kappa\kappa^2_1\kappa^2_P\de_0}{\sqrt{t}}.
	\end{equation}
  This contradicts with the assumption that equality holds in \eqref{estimate:parabolic-h} at $t=t^*$. Hence, by density argument we conclude that \eqref{estimate:parabolic-h} holds for $L^1$ initial data as well.	
\end{proof}	
	Now, we can conclude the bound on $\norm{h_x(t)}_{L^1}$ in terms of $\norm{h_0}_{L^1}$. More precisely, we prove the following result.
\begin{corollary}\label{corollary:h}
	Let $u$ be as in Proposition \ref{prop:parabolic}. Let $h$ be a solution to the equation \eqref{eqn:h} satisfying 
	\begin{equation}\label{assumption:h-2}
		\norm{h(t,\cdot)}_{L^1}\leq \frac{\kappa\norm{h_0}_{L^1}}{2}\leq \de_0\mbox{ for all }t\in[0,\hat{t}]\mbox{ where }\hat{t}:=\left(\frac{1}{2000\kappa_1^4\kappa_A^3\kappa_B^6\kappa_P^{20}\de_0}\right)^2,
	\end{equation}
	for some $\de_0<1$ and $\kappa,\kappa_1,\kappa_A,\kappa_B,\kappa_P$ are as in \eqref{def:kappa}--\eqref{def:kappa-P}. Then we have
	\begin{align}
		\norm{h_{x}(t,\cdot)}_{L^1}&\leq \frac{\kappa^2\kappa_1^2\kappa_P^2\norm{h_0}_{L^1}}{\sqrt{t}}\mbox{ for }t\in(0,\hat{t}].
	\end{align}
\end{corollary}	
\begin{proposition}
Let $T\geq \hat{t}$. Let $u$ and $h$ be the solutions of \eqref{eqn-main} and \eqref{eqn:h} on $[0,T]$ such that that $\| u_x(t,\cdot)\|_{L^1}\leq \delta_0$ and $\|h(t,\cdot)\|_{L^1}\leq\delta_0$ for any $t\in[0,T]$ then we get
\begin{equation}
\begin{cases}
\begin{aligned}
&\|u_{xx}(t,\cdot)\|_{L^1}=O(1)\delta_0^2\;\;\forall t\in[\hat{t},T],\\
&\|h_{x}(t,\cdot)\|_{L^1}=O(1)\delta_0^2\;\;\forall t\in[\hat{t},T].
\end{aligned}
\end{cases}
\label{estimr}
\end{equation}
\label{proplong}
\end{proposition}
\noi The proof follows exactly the same lines as the previous propositions.
	Next result guarantees the assumptions \eqref{assumption:u-L1}, \eqref{assumption:h-L1} as in Proposition \ref{prop:parabolic} and \ref{prop:parabolic-h} respectively. 
\begin{proposition}\label{prop:local-existence}
	Let $u$ and $h$ be solutions of \eqref{eqn-main} and \eqref{eqn:h} respectively such that
	\begin{equation*}
		TV(u_0)\leq \frac{\de_0}{4\kappa}\mbox{ and }\norm{h_0}_{L^1}\leq \frac{\de_0}{4\kappa}.
	\end{equation*}
   Then $u,h$ are well-defined on $[0,\hat{t}]$ where $\hat{t}$ is defined as in \eqref{assumption:u-L1}. Moreover, we have
   \begin{equation}\label{local-estimate-u-h}
   	\norm{u_x(t)}_{L^1}\leq \frac{\de_0}{2}\mbox{ and }\norm{h(t)}_{L^1}\leq 2\kappa \norm{h_0}_{L^1}\mbox{ for }t\in[0,\hat{t}].
   \end{equation}
\end{proposition}
\begin{proof}[Proof of Proposition \ref{prop:local-existence}:]
	Suppose that there exists a time $\tau<\hat{t}$ such that $\norm{u_x(\tau)}_{L^1}= \frac{\de_0}{2}$ and $\norm{u_x(t)}_{L^1}< \frac{\de_0}{2}$ for all $t\in [0,\tau]$. From \eqref{eqn-main}, we can write
	\begin{equation*}
		u_t+A(u^*)u_x=B(u^*)u_{xx}+(B(u)-B(u^*))u_{xx}+(A(u^*)-A(u))u_x+u_x\bullet Bu_x.
	\end{equation*}
Hence, we get
\begin{align*}
	u_x(\tau)&=G(\tau)\star u_{0,x}+\int\limits_{0}^{\tau}G_x(\tau-s)\star\Big\{(B(u)-B(u^*))u_{xx}\\
	&\quad\quad\quad\quad\quad\quad\quad\quad\quad\quad\quad\quad+(A(u^*)-A(u))u_x+u_x\bullet Bu_x\Big\}\,ds.
\end{align*}
Therefore,
\begin{align*}
	\norm{u_x(\tau)}_{L^1}&\leq \kappa\norm{u_{0,x}}_{L^1}+\int\limits_{0}^{\tau}\frac{\kappa}{\sqrt{\tau-s}}\norm{B(u)-B(u^*)}_{L^\f}\norm{u_{xx}}_{L^1}\,ds\\
	&+\int\limits_{0}^{\tau}\frac{\kappa}{\sqrt{\tau-s}}\left[\norm{A(u)-A(u^*)}_{L^\f}\norm{u_{x}}_{L^1}+\kappa_B\norm{u_x}_{L^\f}\norm{u_x}_{L^1}\right]\,ds\\
	&\leq \kappa\norm{u_{0,x}}_{L^1}+\int\limits_{0}^{\tau}\frac{2\kappa\kappa_B}{\sqrt{\tau-s}}\norm{u_x}_{L^1}\norm{u_{xx}}_{L^1}\,ds+\int\limits_{0}^{\tau}\frac{\kappa\kappa_A}{\sqrt{\tau-s}}\norm{u_x}^2_{L^1}.
\end{align*}
By Proposition \ref{prop:parabolic} we get
\begin{equation*}
		\norm{u_{xx}(t,\cdot)}_{L^1}\leq \frac{2\kappa\kappa^2_1\kappa^2_P\de_0}{\sqrt{t}}\mbox{ for }t\in[0,\tau].
\end{equation*}
Hence, by the choice of $\de_0$
\begin{align*}
	\norm{u_x(\tau)}_{L^1}&\leq \frac{\de_0}{4}+2\int\limits_{0}^{\tau}\frac{\kappa}{\sqrt{\tau-s}}\frac{2\kappa\kappa^2_1\kappa^2_P\de_0}{\sqrt{s}}\frac{\kappa_B\de_0}{2}\,ds+\int\limits_{0}^{\tau}\frac{\kappa}{\sqrt{\tau-s}}\frac{\kappa_A\de^2_0}{4}\,ds\\
	&<\frac{\de_0}{2}.
\end{align*}
Hence, we obtain the first inequality in \eqref{local-estimate-u-h}. To get the estimate for $h$ we rewrite \eqref{eqn:h} as follows
\begin{align}
 h_t+ A(u^*)h_x&=B(u^*)h_{xx}+(B(u)-B(u^*))h_{xx}+(A(u^*)-A(u))h_x-h\bullet A(u)u_x\nonumber\\
 &+h\bullet B(u)u_{xx}+(u_x\otimes h):D^2B(u)u_x \nonumber\\
	&+u_x\bullet B(u)h_x+h_x\bullet B(u)u_x  \nonumber\\
	&=B(u^*)h_{xx}+((B(u)-B(u^*))h_x)_x+(h\bullet B(u)u_x)_x  \nonumber\\
	&+(A(u^*)-A(u))h_x-h\bullet A(u)u_x.\label{eqn-h-1}
\end{align}	
We claim that $\norm{h(t)}_{L^1}\leq 2\kappa\norm{h_0}_{L^1}$ for all $t\in[0,\hat{t}]$ where $\hat{t}$ is defined as in \eqref{assumption:u-L1}. Suppose that there exists time $\tau<\hat{t}$ such that $\norm{h(\tau)}_{L^1}=2\kappa\norm{h_0}_{L^1}$ and $\norm{h(t)}_{L^1}\leq 2\kappa\norm{h_0}_{L^1}$ for $t\in[0,\tau]$. From \eqref{eqn-h-1}, we can write
\begin{align*}
	h(\tau)&=G(\tau)\star h_{0}+\int\limits_{0}^{\tau}G_x(\tau-s)\star\Big\{(B(u)-B(u^*))h_x+h\bullet B(u)u_x\Big\}\,ds\\
	&\quad\quad\quad\quad\quad+\int\limits_{0}^{\tau}G(\tau-s)\star\Big\{(A(u^*)-A(u))h_x-h\bullet A(u)u_x\Big\}\,ds.
\end{align*}
Then it follows,
\begin{align*}
	\norm{h(\tau)}_{L^1}&\leq \norm{G(\tau)}_{L^1}\norm{h_0}_{L^1}+\int\limits_{0}^{\tau}\norm{G_x(\tau-s)}_{L^1}\norm{B(u(s))-B(u^*)}_{L^\f}\norm{h_x(s)}_{L^1}\,ds\\
	&+\int\limits_{0}^{\tau}\kappa_B\norm{G_x(\tau-s)}_{L^1}\norm{h(s)}_{L^1}\norm{u_x}_{L^\f}\,ds\\
	&+\int\limits_{0}^{\tau}\norm{G(\tau-s)}_{L^1}\norm{A(u(s))-A(u^*)}_{L^\f}\norm{h_x(s)}_{L^1}\,ds\\
	&+\int\limits_{0}^{\tau}\kappa_A\norm{G(\tau-s)}_{L^1}\norm{h(s)}_{L^1}\norm{u_x(s)}_{L^\f}\,ds\\
	&\leq \kappa \norm{h_0}_{L^1}+\int\limits_{0}^{\tau}\frac{\kappa}{\sqrt{\tau-s}}\frac{\kappa_B\de_0}{2}\frac{\kappa^2\kappa_1^2\kappa_P^2\norm{h_0}}{\sqrt{s}}\,ds\\
	&+2\kappa\norm{h_0}_{L^1}\int\limits_{0}^{\tau}\frac{\kappa_B\kappa}{\sqrt{\tau-s}}\frac{2\kappa\kappa^2_1\kappa^2_P\de_0}{\sqrt{s}}\,ds+\int\limits_{0}^{\tau}\kappa\frac{\kappa_A\de_0}{2}\frac{\kappa^2\kappa_1^2\kappa_P^2\norm{h_0}}{\sqrt{s}}\,ds\\
	&+2\kappa\norm{h_0}_{L^1}\int\limits_{0}^{\tau}\kappa_A\kappa\frac{2\kappa\kappa^2_1\kappa^2_P\de_0}{\sqrt{s}}\,ds\\
	&<2\kappa\norm{h_0}_{L^1}.
\end{align*}
This completes the proof of the last inequality of \eqref{local-estimate-u-h}.
\end{proof}

\section{Interaction estimates}
In this section, we prove the interaction estimates for two different family, what is called as `transversal interaction' in \cite{BiB-vv-lim}. This has been previously studied in \cite{BiB-vv-lim,BiB-triangular,BiB-temple-class} for constant viscosity coefficient. We adapt the results of \cite{BiB-temple-class,BiB-vv-lim} in our set up when viscosity coefficient is depending on space and time. 
\begin{lemma}\label{lemma:transversal-1}
	Let $z,z^\#$ be solutions of the two independent scalar equations,
	\begin{align}
		z_t+(\la(t,x)z)_x-(\mu z)_{xx}&=\varphi(t,x),\label{eqn-z-1}\\
		z^\#_t+(\la^\#(t,x)z^\#)_x-(\mu^\# z^\#)_{xx}&=\varphi^\#(t,x),\label{eqn-z-2}
	\end{align}
	which is valid for $t\in[0,T]$. We assume that
	\begin{equation}\label{relqtion-la-12}
		\inf\limits_{t,x}\la^\#(t,x)-\sup\limits_{t,x}\la(t,x)\geq c>0.
	\end{equation}
	and $\norm{\mu,\mu^\#}_{L^\f}<\f$. Then we have
	\begin{equation}\label{est:transversal-1}
		\int\limits_{0}^{T}\int\limits_{\R}\abs{z(t,x)}\abs{z^\#(t,x)}\,dxdt\leq \frac{1}{c}E_1E_2,
	\end{equation}
	where $E_1,E_2$ are defined as follows
	\begin{align}
		E_1&:=\int\limits_{\R}\abs{z(0,x)}\,dx+\int\limits_{0}^{T}\int\limits_{\R}\abs{\varphi(t,x)}\,dxdt,\\
		E_2&:=\int\limits_{\R}\abs{z^\#(0,x)}\,dx+\int\limits_{0}^{T}\int\limits_{\R}\abs{\varphi^\#(t,x)}\,dxdt.
	\end{align}
\end{lemma}
\begin{proof}[Proof of Lemma \ref{lemma:transversal-1}:]
	Set $c_1:=\norm{\mu,\mu^\#}_{L^\f}$. Let $z,z^\#$ be the solution to \eqref{eqn-z-1}, \eqref{eqn-z-2} with $\varphi=\varphi^\#=0$. 
	Consider
	\begin{equation}
		Q(z,z^\#):=\int\int K(x-y)\abs{z(x)}\abs{z^\#(y)}\,dxdy,
	\end{equation}
	where $K$ is defined as follows
	\begin{equation}
		K(s):=\left\{\begin{array}{rl}
			1/c&\mbox{ if }s\geq 0,\\
			1/c e^{\frac{cs}{2c_1}}&\mbox{ if }s<0.
		\end{array}\right.
	\end{equation}
	Now, we can calculate
	\begin{align*}
		\frac{d}{dt}Q(z(t),z^\#(t))&=\int\int K(x-y)[sgn(z(x))z_t(x)\abs{z^\#(y)}+sgn(z^\#(y))z^\#_t(y)\abs{z(x)}]\,dxdy\\
		&=\int\int K(x-y)\bigg[sgn(z(x))((\mu(x) z(x))_{xx}-(\la z(x))_x)\abs{z^\#(y)}\\
		&\quad\quad+sgn(z^\#(y))((\mu^\#(y) z^\#(y))_{yy}-(\la^\# z^\#(y))_y)\abs{z(x)}\bigg]\,dxdy\\
		&=\int\int K^\p(x-y)\bigg\{\la\abs{z(x)}\abs{z^\#(y)}-\la^\#\abs{z(x)}\abs{z^\#(y)}\bigg\}\,dxdy\\
		&+\int\int K^{\p\p}(x-y)\bigg\{\mu(x)\abs{z(x)}\abs{z^\#(y)}+\mu^\#(y)\abs{z(x)}\abs{z^\#(y)}\bigg\}\,dxdy\\
		&\leq -\int\int (cK^\p(x-y)-2c_1K^{\p\p}(x-y))\abs{z(x)}\abs{z^\#(y)}\,dxdy\\
		&\leq -\int \abs{z(x)}\abs{z^\#(x)}\,dx.
	\end{align*}
	Hence, we get
	\begin{equation}\label{estimate-interaction-1}
		\int_{0}^T\int_{\R} \abs{z(t,x)}\abs{z^\#(t,x)}\,dxdt\leq Q(z(0),z^\#(0))\leq \frac{1}{c}\norm{z(0)}_{L^1}\norm{z^\#(0)}_{L^1}.
	\end{equation}
   Now, we consider $z,z^\#$ as solutions of \eqref{eqn-z-1}, \eqref{eqn-z-2} respectively when $\varphi$ and $\varphi^\#$ may not be identically $0$. Let $\Gamma, \Gamma^\#$ be the fundamental solutions corresponding to the homogeneous system of \eqref{eqn-z-1}--\eqref{eqn-z-2}. Then we can write
   \begin{align}
   	z(t,x)&=\int\limits_{\R}\Gamma(t,x,0,y)z(0,y)\,dy+\int\limits_{0}^{t}\int\limits_{\R}\Gamma(t,x,s,y)\varphi(s,y)\,dyds,\\
   	z^\#(t,x)&=\int\limits_{\R}\Gamma^\#(t,x,0,y)z^\#(0,y)\,dy+\int\limits_{0}^{t}\int\limits_{\R}\Gamma^\#(t,x,s,y)\varphi^\#(s,y)\,dyds,
   \end{align}
   Due to \eqref{estimate-interaction-1}, we have
   \begin{equation*}
   	\int\limits_{\max\{s_1,s_2\}}^{T}\int\limits_{\R}\Gamma(t,x,s_1,y_1)\Gamma^\#(t,x,s_2,y_2)\,dxdt\leq \frac{1}{c},
   \end{equation*}
for any two pairs $(s_1,y_1),(s_2,y_2)\in \R_+\times\R$. Hence, the estimate \eqref{est:transversal-1} follows.
\end{proof}

\begin{lemma}\label{lemma:transversal-2}
	Let $z,z^\#$ be solutions of \eqref{eqn-z-1}, \eqref{eqn-z-2} respectively and we assume that \eqref{relqtion-la-12} holds along with the following estimates
	\begin{align}
		&\int\limits_{0}^{T}\int\limits_{\R}\abs{\varphi(t,x)}dxdt\leq \de_0,\quad\quad 	\int\limits_{0}^{T}\int\limits_{\R}\abs{\varphi^\#(t,x)}dxdt\leq \de_0,\\
		&\norm{z(t)}_{L^1},\,\norm{z^\#(t)}_{L^1}\leq\de_0,\quad\quad \norm{z_x(t)}_{L^1},\,\norm{z^\#(t)}_{L^\f}\leq C_*\de^2_0,\\
		&\norm{\la_x(t)}_{L^\f},\,\norm{\la_x(t)}_{L^1}\leq C_*\de_0,\quad\quad \lim\limits_{x\rr-\f}\la(t,x)=0,
	\end{align}
	for all $t\in[0,T]$. Then we have
	\begin{equation}\label{est:transversal-2}
		\int\limits_{0}^{T}\int\limits_{\R}\abs{z_x(t,x)}\abs{z^\#(t,x)}\,dxdt=\mathcal{O}(1)\de^2_0.
	\end{equation}
\end{lemma}

\begin{proof}[Proof of Lemma \ref{lemma:transversal-2}:]
	We define
	\begin{equation}
		\mathcal{I}_1(T):=\sup\limits_{(\tau,\xi)\in[0,T]\times R}\int\limits_{0}^{T-\tau}\int\limits_{\R}\abs{z_x(t,x)z^\#(t+\tau,x+\xi)}\,dxdt\leq (C^*\de_0^2)^2T.
	\end{equation}
	Let us denote the solution of $z_t=(\mu z_x)_x$ as $G^\mu$. Then $G^\mu$ satisfies the following estimate
	\begin{equation}
		\abs{G^\mu_x(t,x;s,y)}\leq C_1e^{-q_*\frac{(x-y)^2}{t-s}},\mbox{ for some }C_1,q_*>0.
	\end{equation}
	Such $G^\mu$ can be constructed by method of parametrix (see \cite[Theorem 11, Chapter 1]{Friedman}). Therefore, we have 
	\begin{equation}\label{estimate-G-mu}
		\norm{G^\mu_x(1,x;0,x-\cdot)}_{L^1(\R)},\, \norm{G^\mu_x(t,x;t-\cdot,x-\cdot)}_{L^1(0,1;L^1(\R))}\leq C_1^*
	\end{equation} 
	for some $C_1^*>0$. We can write
	\begin{align}
		z_x(t,x)&=\int G^\mu_x(1,x;0,y)z(t-1,y)\,dy\\
		&+\int\limits_{0}^1\int G^\mu_x(1,x;s,y)[\varphi-(\la z)_x](t-1+s,y)\,dyds.
	\end{align}
	By using \eqref{estimate-G-mu} we get
	\begin{align*}
		&\int\limits_{1}^{T-\tau}\int \abs{z_x(t,x)z^\#_x(t+\tau,x+\xi)}\,dxdt\\
		&\leq \int\limits_{1}^{T-\tau}\int\int \abs{G^\mu_x(1,x;0,x-y) z(t-1,x-y)z^\#(t+\ta,x+\xi)}\,dydxdt\\
		&+\int\limits_{1}^{T-\tau}\int\int\limits_{0}^{1}\int \norm{\la_x}_{L^\f}\abs{G^\mu_x(1,x;s,x-y) z(t-1+s,x-y)z^\#(t+\ta,x+\xi)}\,dydsdxdt\\
		&+\int\limits_{1}^{T-\tau}\int\int\limits_{0}^{1}\int \norm{\la}_{L^\f}\abs{G^\mu_x(1,x;s,x-y) z_x(t-1+s,x-y)z^\#(t+\ta,x+\xi)}\,dydsdxdt\\
		&+\int\limits_{1}^{T-\tau}\int\int\limits_{0}^{1}\int \abs{G^\mu_x(1,x;s,x-y) \varphi(t-1+s,x-y)z^\#(t+\ta,x+\xi)}\,dydsdxdt\\
		&\leq \left(\int\abs{G_x^\mu(1,x)}\,dx+\sup\limits_{x}\int\limits_{0}^{1}\int \norm{\la_x}_{L^\f}\abs{G^\mu_x(1,x;s,x-y)}\,dyds\right)\cdot\\
		&\quad\quad\quad\quad\quad\quad\quad\quad\quad\quad\quad\quad\sup\limits_{s,y,\tau,\xi}\Big(\int\limits_{1}^{T-\tau}\abs{ z(t-1+s,x-y)z^\#(t+\ta,x+\xi)}\,dxdt\Big)\\
		&+\left(\sup\limits_{x}\int\limits_{0}^{1}\int \norm{\la}_{L^\f}\abs{G^\mu_x(1,x;s,x-y)}\,dyds\right)\cdot\\
		&\quad\quad\quad\quad\quad\quad\quad\quad\quad\quad\quad\quad\sup\limits_{s,y,\tau,\xi}\Big(\int\limits_{1}^{T-\tau}\abs{ z_x(t-1+s,x-y)z^\#(t+\ta,x+\xi)}\,dxdt\Big)\\
		&+\norm{z^\#}_{L^\f([0,T]\times\R)}\left[\sup\limits_{x}\int\limits_{0}^{1}\int \abs{G^\mu_x(1,x;s,x-y)}\,dyds\right]\cdot \int\limits_{0}^{T}\int\abs{\varphi(t,x)}\,dxdt\\
		&\leq C^*_1(1+\norm{\la_x}_{L^\f})\frac{4\de_0^2}{c}+C_1^*\norm{\la}_{L^\f}\mathcal{I}_1(T)+C_1\de_0^3.
	\end{align*}
	We note that $\mathcal{I}_1(1)\leq C^2_1(\de_0)^2$. Since $C_1^*\norm{\la}_{L^\f}\leq C_1^*\norm{\la_x}_{L^\f}\leq C_1^*\de_0<1/2$, we obtain
	\begin{equation}
		\mathcal{I}_1(T)\leq C^*_1(1+\norm{\la_x}_{L^\f})\frac{4\de_0^2}{c}+C^2_1(\de_0)^2+C_1\de_0^3+\frac{1}{2}\mathcal{I}_1(T).
	\end{equation}
	Hence, we get \eqref{est:transversal-2}.
\end{proof}
\section{BV bounds}\label{sec:BV}
Let us consider an initial data satisfying $TV(u_0)\leq \frac{\de_0}{8\sqrt{n}}$ and $\lim\limits_{x\rr-\f}u(x)=u^*\in K$. Then by applying Proposition \ref{prop:local-existence}, we obtain
\begin{equation}
	\norm{u_x(\hat{t})}_{L^1(\R)}\leq\frac{\de_0}{4\sqrt{n}},
\end{equation}
where $\hat{t}$ is defined as in \eqref{assumption:u-L1}. To get the total variation bound in $(\hat{t},\f)$ we argue by contradiction as in \cite{BiB-vv-lim}. Let $T$ be defined as follows
\begin{equation}\label{def:T-max-time}
	T:=\sup\left\{\tau;\sum\limits_{i}\int\limits_{\hat{t}}^{\tau}\int\limits_{\R}\abs{\phi_i(t,x)}\,dxdt\leq\frac{\de_0}{2}\right\}.
\end{equation}
It $T<\f$, we get a contradiction as follows. From \eqref{def:T-max-time}, we have
\begin{equation}
	\norm{u_x(t)}_{L^1}\leq 2\sqrt{n}\norm{u_x(\hat{t})}_{L^1}+\frac{\de_0}{2}\leq \de_0\mbox{ for all }t\in[\hat{t},T].
\end{equation} 
 By applying Lemma \ref{lemma:transversal-1} and \ref{lemma:transversal-2}, we get 
 \begin{equation}
 	\int\limits_{\hat{t}}^{\tau}\int\limits_{\R}\abs{\phi_i(t,x)}\,dxdt=\mathcal{O}(1)\de_0^2<\frac{\de_0}{2}
 \end{equation}
for sufficiently small $\de_0>0$. Hence, $T$ is not the supremum defined as in \eqref{def:T-max-time}. Hence, $\int\limits_{\hat{t}}^{\tau}\int\limits_{\R}\abs{\phi_i(t,x)}\,dxdt\leq\frac{\de_0}{2}$ for all $t>\hat{t}$. Subsequently, we obtain $\norm{u_x(t)}_{L^1}\leq \de_0$.

\section{Stability estimates}\label{sec:stability}
To prove the stability estimate for viscosity solutions, let us consider the solution of the following linear equation
\begin{align}
	 h_t+h\bullet A(u)u_x+ A(u)h_x
	=&B(u)h_{xx}+h\bullet B(u)u_{xx}+(u_x\otimes h):D^2B(u)u_x\nonumber\\
	&+u_x\bullet B(u)h_x+h_x\bullet B(u)u_x.\label{eqn:h-stabilty}
\end{align}	
By using Proposition \ref{prop:local-existence}, we can have that
\begin{equation}
	\norm{h(\hat{t})}_{L^1}\leq \frac{\norm{h_0}_{L^1}}{2},
\end{equation}
where $\hat{t}$ is defined as in \eqref{assumption:u-L1}. For $t>\hat{t}$ we apply similar technique as in section \ref{sec:BV}. By using interaction estimates (Lemma \ref{lemma:transversal-1} and \ref{lemma:transversal-2}) we can obtain 
\begin{equation}
	\norm{h(t)}_{L^1}\leq \frac{\norm{h(\hat{t})}_{L^1}}{2}\mbox{ for all }t>\hat{t}.
\end{equation}
Combining the above inequality with Proposition \ref{prop:local-existence}, we have 
\begin{equation}\label{stability-estimate-1}
	\norm{h(t)}_{L^1}\leq L_3\norm{h_0}_{L^1}\mbox{ for all }t>0.
\end{equation}
To translate this result to the $L^1$ stability estimate of two viscosity solutions $u,v$ we use the homotopy method as in \cite{BiB-vv-lim,BiB-temple-class}. Let $u^\theta$ be the solution to \eqref{eqn-main} corresponding to the initial data $\theta\bar{u}+(1-\theta)\bar{v}$. Then let $h$ be defined as $h^\theta:=\frac{d u^\theta}{d\theta}$. Then $h^\theta$ solves \eqref{eqn:h-stabilty}. By the estimate \eqref{stability-estimate-1},
\begin{equation*}
	\norm{h^\theta(t)}_{L^1}\leq L_3\norm{h^\theta_0}_{L^1}=L_3\norm{\bar{u}-\bar{v}}_{L^1}\mbox{ for all }t>0.
\end{equation*}
Hence, for all $t>0$ it follows,
\begin{align}
	\norm{u(t)-v(t)}_{L^1}\leq \int\limits_{0}^{1}\norm{\frac{du^\theta(t)}{d\theta}}_{L^1}\,d\theta\leq L_3\norm{\bar{u}-\bar{v}}_{L^1}.
\end{align}
\section{Propagation speed}\label{sec:propagation}

\begin{lemma}\label{lemma:finite-speed}
	There exists $\al_1>0, \B_1>0$ such that the following holds. Let $u,v$ be solutions of \eqref{eqn-main} satisfying \eqref{condition-data-thm-1} and initial data satisfy
	\begin{equation}\label{condition:finite-propagation-1}
		u(0,x)=v(0,x)\mbox{ for }x\notin [a,b].
	\end{equation}
Then for all $x\in\R,\, t>0$
\begin{equation}
			\abs{u(t,x)-v(t,x)}\leq \al_1\norm{u(0)-v(0)}_{L^\f}\cdot\min \left\{e^{c_1(\B_1t-(x-b))},e^{c_1(\B_1t+(x-a))}\right\}.
\end{equation}
Let $\hat{u},\hat{v}$ be solutions of \eqref{eqn-main} satisfying \eqref{condition-data-thm-1} and initial data satisfy
\begin{equation}\label{condition:finite-propagation-2}
	\hat{u}(0,x)=\hat{v}(0,x)\mbox{ for }x\in [a,b].
\end{equation}
Then for all $x\in\R,\, t>0$
\begin{equation}\label{finite-propagation-2}
	\abs{\hat{u}(t,x)-\hat{v}(t,x)}\leq \al_1\norm{\hat{u}(0)-\hat{v}(0)}_{L^\f}\cdot \left(e^{c_1(\B_1t-(x-a))}+e^{c_1(\B_1t+(x-b))}\right).
\end{equation}
\end{lemma}
Before we proceed for the proof of Lemma \ref{lemma:finite-speed}, we recall an estimate from \cite{Friedman}. Let $G^B$ be the fundamental solution to the nonlinear parabolic equation $h_t-(B(u(t,x))h_x)_x=0$. Then we have 
\begin{equation}\label{estimate:G-B}
	\abs{G^B(s,y;t,x)}\leq C_B\frac{e^{-\frac{c_1(x-y)^2}{4(t-s)}}}{\sqrt{t-s}}\mbox{ and }\abs{G_x^B(s,y;t,x)}\leq C_B\frac{e^{-\frac{c_1(x-y)^2}{4(t-s)}}}{t-s},
\end{equation}
for some constant $C_B$.
\begin{proof}[Proof of Lemma \ref{lemma:finite-speed}:]
	Recall the linearized system \eqref{eqn:h},
	\begin{equation}
		h_t-(B(u)h_x)_x=(h\bullet B(u)u_x)_x-(A(u)h)_x+u_x\bullet A(u)h-h\bullet A(u)u_x.
	\end{equation}	
 We consider the integral representation of $h(t,x)$ as follows,
\begin{align}
	h(t,x)=&G^B(t)\star h_0-\int\limits_{0}^{t}G^{B}_x(t-s)\star (A(u(s))h(s))\,ds\nonumber\\
	&+\int\limits_{0}^{t}G^{B}(t-s)\star(u_x\bullet A(u)h-h\bullet A(u)u_x)(s)\,ds.\label{eq:int-representation:h}
\end{align}
We consider initial data $h_0$ as follows
\begin{equation}
	\left\{\begin{array}{rl}
		\abs{h(0,x)}\leq 1 &\mbox{ if }x\leq 0,\\
		h(0,x)=0 &\mbox{ if }x> 0.
	\end{array}\right.
\end{equation}
We would like to estimate $h(t,x)$ on the domain $\{x>bt\}$. Let us consider $\zeta(t)$ satisfying the following estimate
\begin{equation}
	\zeta(t)\geq \hat{C}_B+2\norm{A}_{L^\f}\int\limits_{0}^{t}\left(\frac{1}{\sqrt{t-s}}+\sqrt{\pi}\right) \zeta(s)\,ds\mbox{ and }\zeta(0)=1,
\end{equation}
where $\hat{C}_B=1+\frac{2C_B\sqrt{\pi}}{\sqrt{c_1}}$. Furthermore, let $e(t,x)$ be defined as follows,
\begin{equation}
	e(t,x):=\zeta(t)\,\mbox{exp}\left\{4\hat{C}_B\norm{DA}_{L^\f}\int\limits_{0}^{t}\norm{u_x(s)}_{L^\f}\,ds+t-x\right\}.
\end{equation}
From \eqref{eq:int-representation:h} we can obtain
\begin{align*}
	\abs{h(t,x)}&\leq \int\limits_{\R}\abs{G^B(t,x-y)}\abs{h(0,y)}\,dy\\
	&+\norm{A}_{L^\f}\int\limits_{0}^{t}\int\limits_{\R}\abs{G^B_x(t-s,x-y)}\abs{h(s,y)}\,dyds\\
	&+\norm{DA}_{L^\f}\int\limits_{0}^{t}\int\limits_{\R}\abs{G^B(t-s,x-y)}\abs{u_x(s,y)}\abs{h(s,y)}\,dyds.
\end{align*}
Now, we observe that
\begin{equation}
	\int\limits_{\R}\abs{G^B(t,x-y)}\abs{h(0,y)}\,dy< C_B\int\limits_{\R}\frac{e^{-\frac{c_1(x-y)^2}{4t}}}{\sqrt{ t}}e^{-c_1y}\,dy=\frac{2C_B\sqrt{\pi}}{\sqrt{c_1}}e^{c_1(t-x)},
\end{equation}
and applying \eqref{estimate:G-B} we get
\begin{align*}
		&\norm{A}_{L^\f}\int\limits_{0}^{t}\int\limits_{\R}\abs{G^B_x(t-s,x-y)}\abs{h(s,y)}\,dyds\\
	&\leq C_B\norm{A}_{L^\f}\int\limits_{0}^{t}\int\limits_{\R}\frac{e^{-\frac{c_1(x-y)^2}{4(t-s)}}}{t-s}e(s,y)\,dyds,\\	
	&\norm{DA}_{L^\f}\int\limits_{0}^{t}\int\limits_{\R}\abs{G^B(t-s,x-y)}\abs{u_x(s,y)}\abs{h(s,y)}\,dyds\\
	&\leq C_B\norm{DA}_{L^\f}\int\limits_{0}^{t}\int\limits_{\R}\norm{u_x(s,\cdot)}_{L^\f}\frac{e^{-\frac{c_1(x-y)^2}{4(t-s)}}}{\sqrt{t-s}}e(s,y)\,dyds.
\end{align*}
We calculate
\begin{align*}
	&C_B\norm{A}_{L^\f}\int\limits_{0}^{t}\int\limits_{\R}\frac{e^{-\frac{c_1(x-y)^2}{4(t-s)}}}{t-s}e(s,y)\,dyds\\	
	&=C_B\norm{A}_{L^\f}\int\limits_{0}^{t}\frac{\zeta(s)}{t-s}e^{4\hat{C}_B\norm{DA}_{L^\f}\int\limits_{0}^{s}\norm{u_x(\si)}_{L^\f}\,d\si}\int\limits_{\R}e^{-\frac{c_1(x-y)^2}{4(t-s)}+c_1(s-y)}\,dyds.
\end{align*}
Note that
\begin{align*}
	-\frac{c_1(x-y)^2}{4(t-s)}+c_1(s-y)&=-\frac{c_1(x-y)^2}{4(t-s)}-c_1(t-s)+c_1(x-y)+c_1(t-x)\\
	&=c_1(t-x)-\frac{c_1}{4(t-s)}(x-y-(t-s))^2.
\end{align*}
Therefore, we get
\begin{align*}
		&C_B\norm{A}_{L^\f}\int\limits_{0}^{t}\int\limits_{\R}\frac{e^{-\frac{c_1(x-y)^2}{4(t-s)}}}{t-s}e(s,y)\,dyds\\	
		&=C_B\norm{A}_{L^\f}e^{c_1(t-x)}\int\limits_{0}^{t}\frac{\zeta(s)}{t-s}e^{4\hat{C}_B\norm{DA}_{L^\f}\int\limits_{0}^{s}\norm{u_x(\si)}_{L^\f}\,d\si}\int\limits_{\R}e^{-\frac{c_1}{4(t-s)}(x-y-(t-s))^2}\,dyds\\
		&=C_B\norm{A}_{L^\f}\frac{2\sqrt{\pi}}{\sqrt{c_1}}e^{c_1(t-x)}\int\limits_{0}^{t}\frac{\zeta(s)}{\sqrt{t-s}}e^{4\hat{C}_B\norm{DA}_{L^\f}\int\limits_{0}^{s}\norm{u_x(\si)}_{L^\f}\,d\si}\,ds\\
		&\leq e^{4\hat{C}_B\norm{DA}_{L^\f}\int\limits_{0}^{t}\norm{u_x(\si)}_{L^\f}\,d\si}C_B\norm{A}_{L^\f}\frac{2\sqrt{\pi}}{\sqrt{c_1}}e^{c_1(t-x)}\int\limits_{0}^{t}\frac{\zeta(s)}{\sqrt{t-s}}\\
		&\leq  e^{4\hat{C}_B\norm{DA}_{L^\f}\int\limits_{0}^{t}\norm{u_x(\si)}_{L^\f}\,d\si+c_1(t-x)}\left(\frac{\zeta(t)}{2}-\frac{C_B\sqrt{\pi}}{\sqrt{c_1}}\right)\\
		&\leq \frac{e(t,x)}{2}-\frac{C_B\sqrt{\pi}}{\sqrt{c_1}}e^{c_1(t-x)}.
\end{align*}
Similarly, we calculate
\begin{align*}
		& 2C_B\norm{DA}_{L^\f}\int\limits_{0}^{t}\int\limits_{\R}\norm{u_x(s,\cdot)}_{L^\f}\frac{e^{-\frac{c_1(x-y)^2}{4(t-s)}}}{\sqrt{t-s}}e(s,y)\,dyds\\
		&=2C_B\norm{DA}_{L^\f}\int\limits_{0}^{t}\int\limits_{\R}\norm{u_x(s,\cdot)}_{L^\f}\zeta(s)e^{4\hat{C}_B\norm{DA}_{L^\f}\int\limits_{0}^{s}\norm{u_x(\si)}_{L^\f}\,d\si}\frac{e^{-\frac{c_1(x-y)^2}{4(t-s)}+c_1(s-y)}}{\sqrt{t-s}}\,dyds\\
		&=\frac{4C_B\sqrt{\pi}}{\sqrt{c_1}}e^{c_1(t-x)}\norm{DA}_{L^\f}\int\limits_{0}^{t}\norm{u_x(s,\cdot)}_{L^\f}\zeta(s)e^{4\hat{C}_B\norm{DA}_{L^\f}\int\limits_{0}^{s}\norm{u_x(\si)}_{L^\f}\,d\si}\,dyds\\
		&\leq \frac{C_B\sqrt{\pi}}{\sqrt{c_1}}e^{c_1(t-x)}\zeta(t)\int\limits_{0}^{t}4\norm{DA}_{L^\f}\norm{u_x(s,\cdot)}_{L^\f}e^{4\hat{C}_B\norm{DA}_{L^\f}\int\limits_{0}^{s}\norm{u_x(\si)}_{L^\f}\,d\si}\,dyds\\
		&\leq \frac{C_B\sqrt{\pi}}{\sqrt{c_1}}e^{c_1(t-x)}\frac{\zeta(t)}{\hat{C}_B}\left(e^{4\norm{DA}_{L^\f}\int\limits_{0}^{t}\norm{u_x(\si)}_{L^\f}\,d\si}-1\right)\\
		&=\frac{C_B\sqrt{\pi}}{\hat{C}_B\sqrt{c_1}}\zeta(t)e^{4\hat{C}_B\norm{DA}_{L^\f}\int\limits_{0}^{t}\norm{u_x(\si)}_{L^\f}\,d\si+c_1(t-x)}-\frac{C_B\sqrt{\pi}}{\sqrt{c_1}}e^{c_1(t-x)}\frac{\zeta(t)}{\hat{C}_B}\\
		&\leq \frac{e(t,x)}{2}-\frac{C_B\sqrt{\pi}}{\sqrt{c_1}}e^{c_1(t-x)}.
\end{align*}
In the last line, we have used $\frac{C_B\sqrt{\pi}}{\hat{C}_B\sqrt{c_1}}\leq \frac{1}{2}$ and $\zeta(t)\geq\hat{C}_B$. Hence, if $\abs{h(t,x)}\leq e(t,x)$ for all $t\in [0,\tau)$, then we have $\abs{h(\tau,x)}<e(\tau,x)$. Therefore, $\abs{h(t,x)}\leq e(t,x)$ holds for all $t>0$. From \eqref{estimate:parabolic-1} we have
\begin{equation*}
	\norm{u_x(s)}_{L^\f}\leq \max\left\{\frac{2\kappa\kappa_1^2\kappa_P^2\de_0}{\sqrt{s}},\frac{2\kappa\kappa_1^2\kappa_P^2\de_0}{\sqrt{\hat{t}}}\right\}.
\end{equation*}
Hence, we get
\begin{align}
	\abs{h(t,x)}&\leq \abs{e(t,x)}\leq 2\hat{C}_Be^{C_3t}e^{4\hat{C}_B\norm{DA}_{L^\f}\de_0(2\sqrt{t}+t/\hat{t})}e^{t-x}\leq \al_1e^{c_1(\B_1t-x)}.
\end{align}
Let $h^\#$ be a solution to	\eqref{eqn-h-1} with the following initial data
\begin{equation}
	\left\{\begin{array}{rl}
		\abs{h^\#(0,x)}\leq \vartheta &\mbox{ if }x\leq b,\\
		h^\#(0,x)=0 &\mbox{ if }x> b.
	\end{array}\right.
\end{equation}
for some $\vartheta>0$ and $b\in\R$ then it follows from previous observations that
\begin{equation}
	\abs{h^\#(t,x)}\leq \vartheta \al_1e^{c_1(\B_1t-(x-b))}.
\end{equation}
Similarly, let $h^\ddagger$ be a solution to	\eqref{eqn-h-1} with the following initial data
\begin{equation}
	\left\{\begin{array}{rl}
		\abs{h^\ddagger(0,x)}\leq \vartheta &\mbox{ if }x\geq a,\\
		h^\ddagger(0,x)=0 &\mbox{ if }x< a.
	\end{array}\right.
\end{equation}
for some $\vartheta>0$ and $a\in\R$ then it follows from previous observations that
\begin{equation}
	\abs{h^\ddagger(t,x)}\leq \vartheta \al_1e^{c_1(\B_1t+(x-a))}.
\end{equation}
Rest of the proof follows in a similar way as in \cite[Lemma 12.1]{BiB-vv-lim}. For sake of completeness, we briefly discuss here. Let $u^\theta$ be the solution of \eqref{eqn-main} with initial data $u^\theta(0)=\theta u(0)+(1-\theta)v(0)$. Let $h^\theta$ be the solution of the following Cauchy problem,
\begin{align}
	h^\theta_t+(A(u)h^\theta)_x-(B(u)h^\theta_x)_x&=(h^\theta\bullet B(u)u_x)_x+u_x\bullet A(u)h^\theta-h^\theta\bullet A(u)u_x,\\
	h^\theta(0,x)&=u(0,x)-v(0,x).
\end{align}
If we assume \eqref{condition:finite-propagation-1} we have
\begin{align*}
	\abs{h^\theta(t,x)}&\leq \norm{u(0)-v(0)}_{L^\f}\al_1 e^{c_1(\B_1t-(x-b))},\\
	\abs{h^\theta(t,x)}&\leq \norm{u(0)-v(0)}_{L^\f}\al_1 e^{c_1(\B_1t+(x-a))}.
\end{align*}
Hence, we get
\begin{align*}
	\abs{u(t,x)-v(t,x)}&\leq \int\limits_{0}^{1}\abs{\frac{du^\theta(t,x)}{d\theta}}\,d\theta=\int\limits_{0}^{1}\abs{h^\theta(t,x)}\,d\theta\\
	&\leq \al_1\norm{u(0)-v(0)}_{L^\f}\min \left\{e^{c_1(\B_1t-(x-b))},e^{c_1(\B_1t+(x-a))}\right\}.
\end{align*}
The proof of \eqref{finite-propagation-2} follows in a similar manner as in the proof of \cite[Lemma 12.1]{BiB-vv-lim}. We omit here.
\end{proof}

\section{Vanishing viscosity limit}\label{sec:vv-limit}
As claimed in Theorem \ref{theorem-1} we want to prove vanishing viscosity limit as $\e\rr0$ for the following Cauchy problem
\begin{equation}
	u^\e_t+A(u^\e)u^\e_x=\e(B(u^\e)_x)_x\mbox{ and }u^\e(0,x)=\bar{u}(x).
\end{equation}
Note that $u^\e(t,x)=u(t/\e,x/\e)$ where $u$ solves the following problem with fix viscosity but scaled initial data, 
\begin{equation}
	u_t+A(u)u_x=(B(u)_x)_x\mbox{ and }u(0,x)=\bar{u}(\e x).
\end{equation}
Observe that
\begin{align*}
	TV(\bar{u}(\e \cdot ))&= TV(\bar{u}(\cdot)),\\
	\norm{\bar{u}(\e\cdot )}_{L^1}&=\frac{1}{\e}	\norm{\bar{u}(\cdot )}_{L^1}.
\end{align*}
From section \ref{sec:BV}, \ref{sec:stability} and \ref{sec:propagation} we get
\begin{align*}
	TV(u(t))&\leq L_1 TV(\bar{u}),\\
	\norm{u(t)-v(t)}_{L^1}&\leq \frac{L_2}{\e}\norm{\bar{u}-\bar{v}}_{L^1},\\
	\norm{u(t)-u(s)}_{L^1}&\leq L_3\left(\abs{\frac{t}{\e}-\frac{s}{\e}}+\abs{\frac{\sqrt{t}}{\sqrt{\e}}-\frac{\sqrt{s}}{\sqrt{\e}}}\right),
\end{align*}
if $\bar{u}(x)=\bar{v}(x)$ for $x\in[a,b]$ then we have
\begin{equation*}
	\abs{u(t,x)-v(t,x)}\leq \al_1\norm{\bar{u}-\bar{v}}_{L^\f}\left(e^{c_1(\B_1 t-(x-a/\e))}+e^{c_1(\B_1 t+(x-b/\e))}\right).
\end{equation*}
Therefore, we obtain
\begin{align}
	TV(u^\e(t))&\leq L_1 TV(\bar{u}),\label{est-1}\\
	\norm{u^\e(t)-v^\e(t)}_{L^1}&=\e\norm{u(t)-v(t)}_{L^1}\leq L_2\norm{\bar{u}-\bar{v}}_{L^1},\label{est-2}\\
	\norm{u^\e(t)-u^\e(s)}_{L^1}&\leq L_3\left(\abs{t-s}+\sqrt{\e}\abs{\sqrt{t}-\sqrt{s}}\right),\label{est-3}\\
	\abs{u^\e(t,x)-v^\e(t,x)}&\leq \al_1\norm{\bar{u}-\bar{v}}_{L^\f}\left(e^{\frac{c_1}{\e}(\B_1 t-(x-a))}+e^{\frac{c_1}{\e}(\B_1 t+(x-b))}\right).\label{est-4}
\end{align}
This completes the proof of Theorem \ref{theorem-1}.

The convergence of $u^\e$ as $\e\rr0$ follows from a standard argument with an application of Helly's theorem and the $L^1$ continuity \eqref{est-3}. Indeed, due to the uniform TV estimate \eqref{est-1} by using Helly's theorem we can pass to a the limit (up to a subsequence) for a countable dense set $\{t_n\}$ and then applying $L^1$ continuity we can define the limit function at all time $t>0$. We set
\begin{equation}
	L^1_{loc}-\lim\limits_{k\rr\f}u^{\e_k}(t,\cdot)=:S_t(\bar{u}).
\end{equation}
For $\de_0>0$ and compact set $K\subset\mathcal{U}$, we consider
\begin{equation}
	\mathcal{D}_0:=\left\{u:\R\rr\R^n;\,\,u(-\f)\in K\mbox{ and } TV(u)\leq \de_0\right\}
\end{equation}
By considering a  smaller domain $\mathcal{D}\subset \mathcal{D}_0$ which is positively invariant, we can set $S:\R\times\mathcal{D}\rr\mathcal{D}$. From \eqref{est-3} and \eqref{est-4} we conclude the time continuity and continuous dependence on initial data for $S_t$. To complete the proof of Theorem \ref{theorem-2} we need to show the characterization of $S_t$. To this end, next we show that the limit function of vanishing viscosity process satisfies the finite speed of propagation property.

\begin{lemma}[Finite speed of propagation]\label{lemma:fsp}
	Let $\bar{u},\bar{v}\in\mathcal{D}$. Then there exists $\B_1>0$ such that the following holds for $a,b\in\R$,
	\begin{equation}\label{L1-loc-stability-1}
		\int\limits_{a}^{b}\abs{S_t(\bar{u})-S_t(\bar{v})}\,dx\leq L_4	\int\limits_{a-\B_1t}^{b+\B_1t}\abs{\bar{u}-\bar{v}}\,dx.
	\end{equation}
	
\end{lemma}
\begin{proof}
	We first recall that for any $\bar{u},\bar{v}\in\mathcal{D}$ we have 
	\begin{equation*}
		\norm{S_t(\bar{u})-S_t(\bar{v})}_{L^1}\leq L_1\norm{\bar{u}-\bar{v}}_{L^1}\mbox{ for }t>0.
	\end{equation*}
	Note that
	\begin{equation*}
		\norm{S_t(\bar{u})-S_t(\bar{v})}_{L^1}=\sup\limits_{r>0}\int\limits_{-r}^{r}\abs{S_t(\bar{u})-S_t(\bar{v})}\,dx.
	\end{equation*}
	Consider the data $\bar{w}$ defined as 
	\begin{equation*}
		\bar{w}(x)=\left\{
		\begin{array}{rl}
			\bar{u}&\mbox{ for }x\in[a-\B_1t,b+\B_1 t],\\
			\bar{v}&\mbox{ for }x\notin [a-\B_1t,b+\B_1t].
		\end{array}\right.
	\end{equation*}
	Then we have
	\begin{equation}
		\int\limits_{a}^{b}\abs{S_t(\bar{w})-S_t(\bar{v})}\,dx\leq L_1\norm{\bar{w}-\bar{v}}_{L^1}=\int\limits_{a-\B_1t}^{b+\B_1t}\abs{\bar{u}-\bar{v}}\,dx.
	\end{equation}
	To complete the proof it remains to check that $S_t(\bar{u})=S_t(\bar{w})$ on $[a,b]$. To this end we apply \eqref{finite-propagation-2} for $a_1=a-\B_1 t,b_1=b+\B_1t$ to get
	\begin{align*}
		\abs{u^{\e_m}(t,x)-w^{\e_m}(t,x)}&\leq \al_1\norm{\bar{u}(0)-\bar{w}(0)}_{L^\f}\cdot \left(e^{c_1\frac{\B_1t-(x-a_1)}{\e_m}}+e^{c_1\frac{\B_1t+(x-b_1)}{\e_m}}\right)\\
		&\leq \al_1\norm{\bar{u}(0)-\bar{w}(0)}_{L^\f}\cdot \left(e^{c_1\frac{a-x}{\e_m}}+e^{c_1\frac{x-b}{\e_m}}\right).
	\end{align*}
	For $x\in[a,b]$, passing to the limit as $\e\rr0$, we can obtain
	\begin{equation*}
		S_t(\bar{u})=S_t(\bar{w})\mbox{ on }[a,b].
	\end{equation*}
	This completes the proof of Lemma \ref{lemma:fsp}.
\end{proof}
%
%
%
%
%
Finally, we can complete the proof of Theorem \ref{theorem-2}.
\begin{proof}[Proof of Theorem \ref{theorem-2}:]
	We first consider the Riemann data where $u_-,u_+$ both lie on $i$-rarefaction curve. Since the rarefaction curves are straight lines we can write 
	\begin{equation*}
		\bar{u}(x)=u^*+\bar{z}(x)r_i(u^*)\mbox{ where }u^*=u(-\f).
	\end{equation*}
	Consider the flux $F_i$ defined as in \eqref{def:F-i}. Then we note that since the solution $u^\e$ is satisfying \eqref{eqn-parabolic}, we obtain
	\begin{equation}
		z^\e_t+F(z^\e)_x=\e (\mu_i(u^\e)z^\e_x)_x\mbox{ where }u^\e=u^*+z^\e r_i(u^*)\mbox{ and }z(0,x)=\bar{z}(x).
	\end{equation}
	Due to uniform parabolicity, global solution $z^\e$ exists and $z^\e$ converges to entropy solution $z$ of \eqref{eqn:z_i} (see \cite{Kruzkov}). Since the rarefaction curves are straight lines we can write $u^\e=u^*+z^\e(t,x) r_i(u^*)$. Hence, the limit $u$ can be written $u(t,x)=u^*+z(t,x) r_i(u^*)$, in other words, $u^\e$ converges to a solution of \eqref{eqn:hyperbolic}, $u$ defined as in \eqref{soln:Rie}.
	
	We consider an initial data which is perturbation of a Riemann data $\bar{u}_{Rie}=u_-\chi_{(-\f,0)}+u_+\chi_{(0,\f)}$ defined as follows
	\begin{equation}
		\bar{u}(x):=\left\{\begin{array}{rl}
			u_-&\mbox{ if }x<\de,\\
			w_i&\mbox{ if }i\de<x<(i+1)\de,\mbox{ with }1\leq i\leq n-1,\\
			u_+&\mbox{ if }x>n\de.
		\end{array}\right.
	\end{equation}
	Due to finite speed of propagation, up to a small time $t_0$, the waves do not interact with each other and the solution to \eqref{eqn:hyperbolic} can be written as 
	\begin{equation}
		u^\de(t,x)=\mathcal{R}_i(z_i(t,x-i\de);w_{i-1})\mbox{ for }x\in[i\de+\hat{\la}t,(i+1)\de-\hat{\la}t]\mbox{ when }t\in[0,t_0].
	\end{equation}
	Let $u^{\e,\de}$ be the solution to \eqref{eqn-main} corresponding to the initial data $\bar{u}^\de$. By our previous analysis and $L^1_{loc}$-dependence \eqref{L1-loc-stability-1}, we obtain that $u^{\e,\de}$ converges to $u^\de$. Due to strict hyperbolicity, we note that the waves of $u^\de$ has disjoint support of length at least $ct_0$. By repeat the previous argument, we can obtain that $u^{\e,\de}$ converges to $u^\de$ for all $t>0$. Now, by sending $\de\rr0$, we can obtain that the viscosity solution $u^\e$ converges to solution of \eqref{eqn:hyperbolic} for Riemann data. Since a Lipschitz continuous semigroup is determined by the local in time behavior for piecewise constant data, this characterizes the limit function $u$. Moreover, it also says that for any subsequence $\e_k\rr0$, $u^{\e_k}$ converges to the same limit. This completes the proof of Theorem \ref{theorem-2}.  
\end{proof}

\end{document}